\documentclass[a4,11pt]{article}
\usepackage[subrefformat=parens]{subcaption}
\usepackage[dvipdfmx]{graphicx}
\usepackage{amssymb}
\usepackage{amstext}
\usepackage{amsmath}
\usepackage{amscd}
\usepackage{amsthm}
\usepackage{amsfonts}
\usepackage{enumerate}
\usepackage{graphicx}
\usepackage{latexsym}
\usepackage{mathrsfs}
\usepackage{mathtools}
\usepackage{cases}
\usepackage[svgnames]{xcolor}
\usepackage{verbatim}
\usepackage{bm}
\usepackage{tikz}
\usepackage{caption}
\usepackage[top=3.5cm, bottom=3.5cm, left=3.5cm, right=3.5cm]{geometry}
\usepackage[labelformat=empty,labelsep=none]{subcaption}
\usepackage{authblk}
\usepackage{comment}
\usepackage{hyperref}
\newtheorem{thm}{Theorem}[section]

\newtheorem{lem}[thm]{Lemma}

\theoremstyle{definition}
\newtheorem{rem}[thm]{Remark}

\newtheorem*{claim*}{Claim}
\theoremstyle{remark}

\numberwithin{equation}{section}

\title{\vspace{-3cm}\textbf{Inverse medium scattering problems with Kalman filter techniques I. Linear case}}

\author[1]{\rm Takashi Furuya}
\author[2]{\rm Roland Potthast}
\affil[1]{{\small Department of Mathematics, Hokkaido University, Japan}}
\affil[]{{\small Email: takashi.furuya0101@gmail.com}\vspace{3mm}}

\affil[2]{{\small Data Assimilation Unit, Deutscher Wetterdienst, Germany}}
\affil[]{{\small Email: Roland.Potthast@dwd.de}}

\date{}
\begin{document}
\maketitle
\begin{abstract}
In this paper, we study the inverse acoustic medium scattering problem to reconstruct the unknown inhomogeneous medium from far field patterns of scattered waves.
We propose the reconstruction scheme based on the Kalman filter, which becomes possible to sequentially estimate the inhomogeneous medium. 
We also show that in the linear inverse problem, the estimation for the Kalman filter is equivalent to that for the Tikhonov regularization.
Finally, we give numerical examples to demonstrate our proposed method.
\end{abstract}
\date{{\bf Key words}. Inverse acoustic scattering, Inhomogeneous medium, Far field pattern, Tikhonov regularization, Kalman filter.}
\section{Introduction}
The inverse scattering problem is the problem to determine unknown scatterers by measuring scattered waves that is generated by sending incident waves far away from scatterers. 
It is of importance for many applications, for example medical imaging, nondestructive testing, remote exploration, and geophysical prospecting. 
Due to many applications, the inverse scattering problem has been studied in various ways.
For further readings, we refer to the following books \cite{Cakoni, Chen, ColtonKress, Kirsch, NakamuraPotthast}, which include the summary of classical and recent progress of the inverse scattering problem.
\par
We begin with the mathematical formulation of the scattering problem. 
Let $k>0$ be the wave number, and let $\theta \in \mathbb{S}^{1}$ be incident direction. 
We denote the incident field $u^{inc}(\cdot, \theta)$ with the direction $\theta$ by the plane wave of the form 
\begin{equation}
u^{inc}(x, \theta):=\mathrm{e}^{ikx \cdot \theta}, \ x \in \mathbb{R}^2. \label{1.1}
\end{equation} 
Let $Q$ be a bounded domain and let its exterior $\mathbb{R}^2\setminus  \overline{Q}$ be connected. 
We assume that $q \in L^{\infty}(\mathbb{R}^2)$, which refers to the inhomogeneous medium, satisfies $\mathrm{Im}q \geq 0$, and its support $\mathrm{supp}\ q$ is embed into $Q$, that is $\mathrm{supp}\ q \Subset Q$. 
Then, the direct scattering problem is to determine the total field $u=u^{sca}+u^{inc}$ such that
\begin{equation}
\Delta u+k^2(1+q)u=0 \ \mathrm{in} \ \mathbb{R}^2, \label{1.2}
\end{equation}
\begin{equation}
\lim_{r \to \infty} \sqrt{r} \biggl( \frac{\partial u^{sca}}{\partial  r}-iku^{sca} \biggr)=0, \label{1.3}
\end{equation}
where $r=|x|$. 
The {\it Sommerfeld radiation condition} (\ref{1.3}) holds uniformly in all directions $\hat{x}:=\frac{x}{|x|}$. 
Furthermore, the problem (\ref{1.2})--(\ref{1.3}) is equivalent to the {\it Lippmann-Schwinger integral equation}
\begin{equation}
u(x, \theta)=u^{inc}(x, \theta)+k^2\int_{Q}q(y)u(y, \theta)\Phi(x,y)dy, \label{1.4}
\end{equation}
where $\Phi(x,y)$ denotes the fundamental solution to Helmholtz equation in $\mathbb{R}^2$, that is, 
\begin{equation}
\Phi(x,y):= \displaystyle \frac{i}{4}H^{(1)}_0(k|x-y|), \ x \neq y, \label{1.5}
\end{equation}
where $H^{(1)}_0$ is the Hankel function of the first kind of order one. 
It is well known that there exists a unique solution $u^{sca}$ of the problem (\ref{1.2})--(\ref{1.3}), and it has the following asymptotic behaviour,
\begin{equation}
u^{sca}(x, \theta)=\frac{\mathrm{e}^{ikr}}{\sqrt{r}}\Bigl\{ u^{\infty}(\hat{x},\theta)+O\bigl(1/r \bigr) \Bigr\} , \ r \to \infty, \ \ \hat{x}:=\frac{x}{|x|}. \label{1.6}
\end{equation}
The function $u^{\infty}$ is called the {\it far field pattern} of $u^{sca}$, and it has the form
\begin{equation}
u^{\infty}(\hat{x},\theta)=\frac{k^2}{4\pi}\int_{Q}\mathrm{e}^{-ik \hat{x} \cdot y} u(y, \theta)q(y)dy=:\mathcal{F}_{\theta}q(\hat{x}), \label{1.7}
\end{equation}
where the far field mapping $\mathcal{F}_{\theta}:L^{2}(Q) \to L^{2}(\mathbb{S}^{1})$ is defined in the second equality for each incident direction $\theta \in \mathbb{S}^{1}$. 
For further details of these direct scattering problems, we refer to Chapter 8 of \cite{ColtonKress}. 
\par
We consider the inverse scattering problem to reconstruct the function $q$ from the far field pattern $u^{\infty}(\hat{x}, \theta_n)$ for all directions $ \hat{x} \in \mathbb{S}^{1}$ and several directions $\{ \theta_n \}_{n=1}^{N}\subset \mathbb{S}^{1}$ with some $N \in \mathbb{N}$, and one fixed wave number $k>0$.
It is well known that the function $q$ is uniquely determined from the far field pattern $u^{\infty}(\hat{x}, \theta)$ for all $ \hat{x}, \theta \in \mathbb{S}^{1}$ and one fixed $k>0$ (see, e.g., \cite{bukhgeim2008recovering, novikov1988multidimensional, ramm1988recovery}), but the uniqueness for several incident plane wave is an open question. 
For impenetrable obstacle scattering case, if we assume that the shape of scatterer is a polyhedron or ball, then the uniqueness for a single incident plane wave is proved (see \cite{alessandrini2005determining, cheng2003uniqueness, liu2006uniqueness, liu1997inverse}). 
Recently in \cite{alberti2020infinitedimensional}, they showed the Lipschitz stability for inverse medium scattering with finite measurements $\{ u^{\infty}(\hat{x}_{i}, \theta_{j}) \}_{i,j=1,...,N}$ for large $N \in \mathbb{N}$ under the assumption that the true function belongs to some compact and convex subset of finite-dimensional subspace. 
\par
Our problem for equation (\ref{1.7}) with finite measurements $\{ u^{\infty}(\cdot, \theta_n) \}_{n=1}^{N}$ is not only ill-posed, but also nonlinear, that is, the far field mappings $\mathcal{F}_{\theta}$ is nonlinear because $u(\cdot, \theta)$ in (\ref{1.7}) is a solution for the Lippmann-Schwinger integral equation (\ref{1.4}), which depends on $q$. 
Existing methods for solving nonlinear inverse problem can be roughly categorized into two groups: iterative optimization methods and qualitative methods. 
The iterative optimization method (see e.g., \cite{Bakushinsky, ColtonKress, Giorgi, Hohage, Kaltenbacher}) does not require many measurements, however it require the initial guess which is the starting point of the iteration. 
It must be appropriately chosen by a priori knowledge of the unknown function $q$, otherwise, the iterative solution could not converge to the true function. 
On the other hand, the qualitative method such as the linear sampling method \cite{ColtonKirsch}, the no-response test \cite{Honda}, the probe method \cite{Ikehata}, the factorization method \cite{KirschGrinberg}, and the singular sources method \cite{Potthast}, does not require the initial guess and it is computationally faster than the iterative method. 
However, the disadvantage of the qualitative method is to require uncountable many measurements. 
For the survey of the qualitative method, we refer to \cite{NakamuraPotthast}.
Recently in \cite{ito2012direct, Liu_2018}, they suggested the reconstruction method from a single incident plane wave although the rigorous justifications are lacked.
\par
If the total field $u$ in (\ref{1.7}) is replaced by the incident field $u^{inc}$, the nonlinear equation (\ref{1.7}) is transformed into the linear equation
\begin{equation}
u^{\infty}_{B}(\hat{x},\theta)=\frac{k^2}{4\pi}\int_{Q}\mathrm{e}^{-ik \hat{x} \cdot y} u^{inc}(y, \theta)q(y)dy=:\mathcal{F}_{B, \theta}q(\hat{x}), \label{1.8}
\end{equation}
which is known as the {\it Born approximation}. 
The function $u^{\infty}_{B}$ is a good approximation of the far field pattern $u^{\infty}$ when $k>0$ and the value of $q$ are very small (see (\ref{1.4})). 
Another interpretation is that the Born approximation is the Fr\'echet derivative of the far field mapping $\mathcal{F}$ at $q=0$. 
For further readings of the inverse scattering problem with the Born approximation, we refer to \cite{Bakushinsky, Bao, ColtonKress, Kirsch2, Pike}. 
In this paper, we study the linear integral equation (\ref{1.8}) instead of the nonlinear one (\ref{1.7}).
This paper is the first part of our works, and in the forthcoming paper, we will study the nonlinear integral equation (\ref{1.7}). 
\par
Although the inverse scattering problem become linear by the Born approximation, the linear equation (\ref{1.8}) is ill-posed, which means that there does not generally exist the inverse $\mathcal{F}^{-1}_{B, \theta}$ of the operator $\mathcal{F}_{B, \theta}$. 
A common technique to solve linear and ill-posed inverse problems is the {\it Tikhonov regularization method} (see e.g., \cite{Cakoni, Hanke, Kress, NakamuraPotthast}). 
A natural approach applying regularization method to our situation is to put all available measurements $\{ u^{\infty}_{B}(\cdot, \theta_n) \}_{n=1}^{N}$ and all far field mappings $\{ \mathcal{F}_{B, \theta_n} \}_{n=1}^{N}$ into one long vectors $\vec{u}^{\infty}_{B}$ and $\vec{\mathcal{F}}_{B}$, respectively, and to apply the Tikhonov regularization method to the big system equation $\vec{u}^{\infty}=\vec{\mathcal{F}}_{B}q$. 
We shall call this way the {\it Full data Tikhonov}.
\par
In this paper, we propose the reconstruction scheme based on {\it Kalman filter}. 
The Kalman filter (see the original paper \cite{Kalman}) is the algorithm to estimate the unknown state in the dynamics system by using the time sequential measurements. 
It has many applications such as navigations and tracking objects, and for further readings, we refer to \cite{Grewal, Jazwinski, Kalman, NakamuraPotthast}. 
\par
The contributions of this paper are the following.
\begin{itemize}
  \item[(A)] We propose the reconstruction algorithm for solving the linear inverse scattering problem (\ref{1.8}) based on the Kalman Filter (see (\ref{4.21})--(\ref{4.23})).
  \item[(B)] We show that in the linear problem, the Full data Tikhonov is equivalent to the Kalman Filter (see Theorem \ref{equivalence}).
\end{itemize}
(A) means that we can estimate the unknown function $q$ by updating every time to give the far field pattern $u^{\infty}_{B}(\cdot, \theta_{n})$ with one incident direction $\theta_n$ without waiting for all measurements $\{ u^{\infty}_{B}(\cdot, \theta_{n}) \}_{n=1}^{N}$.
Furthermore, (B) means that the final solution of the Kalman filter coincides with the solution $q^{FT}_{N}$ of the Full data Tikhonov when the same initial guess is employed.
The advantage of the Kalman filter over the Full data Tikhonov is that we do not require to construct the big system equation $\vec{u}^{\infty}_{B}=\vec{\mathcal{F}}_{B}q$, which reduces computational costs. 
Instead, we update not only state, but also the weight of the norm for the state space, which is associated with the update of the covariance matrices of the state in the statistical viewpoint (see Section \ref{Stochastic viewpoints of Kalman filter}).
\par
This paper is organized as follows. In Section \ref{Tikhonov regularization method}, we briefly recall the Tikhonov regularization theory. In Sections \ref{Full data Tikhonov}, we give the algorithm of the Full data Tikhonov. In Section \ref{Kalman Filter Section}, we give the algorithm of the Kalman filter, and show that it is equivalent to the Full data Tikhonov. In section \ref{Stochastic viewpoints of Kalman filter}, we discuss the stochastic viewpoints of Kalman filter. Finally in Section \ref{Numerical examples}, we give numerical examples to demonstrate our theoretical results.
\section{Tikhonov regularization method}\label{Tikhonov regularization method}
Tikhonov regularization is the method to provide the stable approximate solution for linear and ill-posed inverse problem. 
In this section, we briefly recall the regularization approach.
For further readings, 
we refer to e.g., \cite{Cakoni, Hanke, Kress, NakamuraPotthast}. 
In Sections 2--5, we consider the general functional analytic situation of our inverse scattering problem.
\par
Let $X$ and $Y$ be Hilbert spaces over complex variables $\mathbb{C}$, which are associated with the state space $L^{2}(Q)$ of the inhomogeneous medium function $q$, and the observation space $L^{2}(\mathbb{S}^1)$ of the far field pattern $u^{\infty}$, respectively, and let $A:X \to Y$ be a compact linear operator from $X$ to $Y$, which is associated with the observation operator $\mathcal{F}_{B}:L^{2}(Q) \to L^{2}(\mathbb{S}^{1})$ defined in (\ref{1.8}) as the far field mapping.
We consider the following problem to determine $\varphi \in X$ given $f \in Y$.
\begin{equation}
A\varphi = f. \label{2.1}
\end{equation} 
Since the observation operator $A$ is not generally invertible, the equation (\ref{2.1}) is replaced by
\begin{equation}
\alpha \varphi + A^{*}A\varphi = A^{*}f, \label{2.2}
\end{equation}
which was derived from the multiplication with the adjoint $A^{*}$ of the operator $A$ and the addition of $\alpha \varphi$ where the regularization parameter $\alpha>0$ in (\ref{2.1}). 
We call the solution $\varphi_{\alpha}$ of the equation (\ref{2.2}) the regularized solution of (\ref{2.1}). 
The following lemma is well known as the properties of the regularized solution $\varphi_{\alpha}$ (see e.g., Chapter 4 of \cite{ColtonKress}, Section 4 of \cite{Groetsch}, and Chapter 3 of \cite{NakamuraPotthast}). 
\begin{lem}\label{lemma Tikhonov}
Let $X$ and $Y$ be Hilbert spaces and let $A:X \to Y$ be a compact linear operator from $X$ to $Y$. Then, followings hold.
\begin{description}
\item[(i)] (Theorems 4.13 in \cite{ColtonKress}) The operator ($\alpha I + A^{*}A$) is bounded invertible.

\item[(ii)] (Theorem 4.14 in \cite{ColtonKress}) There exists a unique $\varphi_{\alpha}$ such that
\begin{equation}
\alpha \left\| \varphi \right\|^{2}_{X} + \left\| f - A\varphi \right\|^{2}_{Y} = \mathrm{inf}_{\varphi \in X}\left\{\alpha \left\| \varphi \right\|^{2}_{X} + \left\| f - A\varphi \right\|^{2}_{Y}  \right\}. \label{2.4}
\end{equation} 
The minimizer $\varphi_{\alpha}$ is given by the unique solution of (\ref{2.2}) which has the form
\begin{equation}
\varphi_{\alpha}= (\alpha I + A^{*}A)^{-1}A^{*}f,
\end{equation}
and depends continuously on $f$.

\item[(iii)] (Lemma 3.2.2 in \cite{NakamuraPotthast} and Section 4.3 of \cite{Groetsch}) Let $X$ be finite-dimensional.
Then, we have
\begin{equation}
\varphi_{\alpha} \to A^{\dag}f, \ \alpha \to 0,
\end{equation}
if $f \in \mathrm{R}(A)$ where the operator $A^{\dag}$ is the pseudo inverse of the operator $A$ defined by $A^{\dag}:=(A^{*}A)^{-1}A^{*}$.
Furthermore, $A^{\dag}f$ is the least squares solution, which is minimizer of the following problem
\begin{equation}
\left\| A \varphi - f \right\|=\mathrm{min}_{\varphi \in X}\left\{ \left\| A \varphi -f \right\|_{Y} \right\}. \label{2.7}
\end{equation}

\item[(iv)] (Theorem 3.1.8 in \cite{NakamuraPotthast}) Let $A$ be injective, and let $f$ be of the form $f=A\varphi^{*}$. 
Then, we have
\begin{equation}
\varphi_{\alpha} \to \varphi^{*}, \ \alpha \to 0.
\end{equation}

\item[(v)] (Theorem 3.1.10 in \cite{NakamuraPotthast})  Let $A$ be injective.
If $f \in \mathrm{R}(A)$, then there exists $C=C_f$ such that
\begin{equation}
\left\| \varphi_{\alpha} \right\| \leq C, \ \alpha>0, \label{2.5}
\end{equation}
and if $f \notin \mathrm{R}(A)$, then $\left\| \varphi_{\alpha} \right\|_{X} \to \infty$ as $\alpha \to 0$.
\end{description}
\end{lem}
\begin{rem}\label{Remark for Tikhonov}
We observe from (iii) that if $X$ is finite-dimensional and $f=A\varphi_{true}$ where $\varphi_{true}$ is the true solution of (\ref{2.1}), the regularized solution $\varphi_{\alpha}$ converges to the least squares solution $A^{\dag}A\varphi_{true}$. 
We remark that the operator $A^{\dag}A$ is an orthogonal projection onto $\mathrm{R}(A^{*})=\mathrm{N}(A)^{\bot}$ (see Lemma 3.2.3 in \cite{NakamuraPotthast}).
Therefore, in addition if the operator $A$ is injective, then the least squares solution $A^{\dag}A\varphi_{true}$ coincides with the true solution $\varphi^{true}$.
\end{rem}
\section{Full data Tikhonov}\label{Full data Tikhonov}
The natural approach for solving the equation (\ref{1.8}) is to put all available measurements $\{ u^{\infty}_{B, n} \}_{n=1}^{N}$ and all far field mappings $\{ \mathcal{F}_{B, n} \}_{n=1}^{N}$, where the index $n$ is associated with some incident angle $\theta_n \in \mathbb{S}^{1}$, into one long vector $\vec{u}^{\infty}_{B}$ and $\vec{\mathcal{F}}_{B}$, respectively, and to employ the regularized approach discussed in the Section 2. 
In order to study the above general situation, let $f_1,..., f_N \in Y$ be measurements, let $A_1,...,A_N$ be observation operators, and let us consider the problem to determine $\varphi \in X$ such that 
\begin{equation}
A_n \varphi = f_n, \label{3.1}
\end{equation}
for all $n=1,...,N$. 
Now, we assume that we have the initial guess $\varphi_0 \in X$, which is the starting point of the algorithm, and  is appropriately determined by a priori information of the true solution $\varphi^{true}$. 
Then, we consider the minimization problem of the following functional.
\begin{eqnarray}
J_{Full, N}(\varphi)&:=&\alpha \left\| \varphi - \varphi_0 \right\|^{2}_{X} + \left\| \vec{f} - \vec{A}\varphi \right\|^{2}_{Y^{N},  R^{-1}}
\nonumber\\
&=&\alpha \left\| \varphi - \varphi_0 \right\|^{2}_{X} + \sum_{n=1}^{N}\left\| f_n - A_n\varphi \right\|^{2}_{Y, R^{-1}}, \label{3.2}
\end{eqnarray}
where  $\vec{f}:=\left(
    \begin{array}{cc}
      f_1 \\
      \vdots \\
      f_N
    \end{array}
  \right)$, and $\vec{A}:=\left(
    \begin{array}{cc}
      A_1 \\
      \vdots \\
      A_N
    \end{array}
  \right)$. 
The norm $\left\| \cdot \right\|^{2}_{Y, R^{-1}}:=\langle \cdot, R^{-1} \cdot \rangle_{Y}$ is a weighted norm with a positive definite symmetric invertible operator $R: Y \to Y$, which is interpreted as the covariance matrices of the observation error distribution from a statistical viewpoint in the case when $Y$ is the Euclidean space (see Section \ref{Stochastic viewpoints of Kalman filter}). 
With $\tilde{\varphi}=\varphi-\varphi_0$, the problem (\ref{3.1}) is transformed into 
\begin{equation}
\tilde{J}_{Full, N}(\tilde{\varphi}):=\alpha \left\| \tilde{\varphi} \right\|^{2}_{X} + \left\| (\vec{f} - \vec{A}\varphi_0 ) - \vec{A}\tilde{\varphi} \right\|^{2}_{Y^{N}}. \label{3.3}
\end{equation}
By Lemma 2.1, the minimizer $\tilde{\varphi}_{\alpha}$ of (\ref{3.3}) is given by
\begin{equation}
\tilde{\varphi}_{\alpha} = (\alpha I + \vec{A}^{*}\vec{A})^{-1}\vec{A}^{*}\left(\vec{f} - \vec{A}\varphi_0 \right), \label{3.4}
\end{equation}
which implies that 
\begin{equation}
\varphi^{FT}_{N} := \varphi_0 + (\alpha I + \vec{A}^{*}\vec{A})^{-1}\vec{A}^{*}\left(\vec{f} - \vec{A}\varphi_0 \right), \label{3.5}
\end{equation}
is the minimizer of (\ref{3.2}). 
We call this the {\it Full data Tikhonov}. 
Here, $\vec{A}^{*}$ is the adjoint operator with respect to $\langle \cdot, \cdot \rangle_{X}$ and $\langle \cdot, \cdot \rangle_{Y^{N}, R^{-1}}$.
We calculate
\begin{eqnarray}
\langle \vec{f}, \vec{A} \varphi \rangle_{Y^N, R^{-1}} &=& \sum_{n=1}^{N} \langle f_n, R^{-1}A_n \varphi \rangle_{Y}
\nonumber\\
&=&\sum_{n=1}^{N} \langle A^{H}_n R^{-1} f_n, \varphi \rangle_{X} = \langle  \vec{A}^{H} R^{-1} \vec{f}, \varphi \rangle_{X} \label{3.6}
\end{eqnarray}
which implies that 
\begin{equation}
\vec{A}^{*}=\vec{A}^{H} R^{-1} \label{3.7}
\end{equation}
where $A_{n}^{H}$ and $\vec{A}^{H}$ are the adjoint operator with respect to usual scalar products $\langle \cdot, \cdot \rangle_{X}$, $\langle \cdot, \cdot \rangle_{Y}$ and $\langle \cdot, \cdot \rangle_{X}$, $\langle \cdot, \cdot \rangle_{Y^{N}}$, respectively.
Then, the Full data Tikhonov solution in (\ref{3.5}) is of the form
\begin{equation}
\varphi^{FT}_{N} = \varphi_0 + \left( \alpha I + \vec{A}^{H} R^{-1}\vec{A} \right)^{-1}\vec{A}^{H} R^{-1}\left(\vec{f} - \vec{A}\varphi_0 \right). \label{3.8}
\end{equation}
\par
However, the solution (\ref{3.8}) of the Full data Tikhonov is computationally expensive when the number $N$ of measurements is increasing in which we have to construct the bigger system $\vec{A}\varphi= \vec{f}$.
So, let us consider the alternative approach based on the Kalman filter in the next section. 
\section{Kalman filter}\label{Kalman Filter Section}
The Kalman filter is the algorithm to estimate the unknown state in the dynamics system by using the sequential measurements over time.
In the usual Kalman filter, the model operator to describe the process of the state in the dynamics system is defined (see e.g., Chapter 5 of \cite{NakamuraPotthast}). 
In our problem, it corresponds to the identity mapping because unknown function $q$ does not develop over time. 
\par
Let us formulate the Kalman filter algorithm based on the functional analytic situation using the same notations described in Sections \ref{Tikhonov regularization method} and \ref{Full data Tikhonov}.
In \cite{Freitag, NakamuraPotthast}, the similar arguments of the following was discussed in the special case when $X$ and $Y$ are the Euclidean spaces. 
In this section, we discuss more general situation, that is, the Hilbert space over complex variables $\mathbb{C}$, which is applicable to our inverse scattering problem.
\par
First, we consider the following minimization problem when one measurement $f_1 \in Y$, observation operator $A_1$, and the initial guess $\varphi_0 \in X$ are given.
\begin{equation}
J_{1}(\varphi):=\alpha \left\| \varphi - \varphi_0 \right\|^{2}_{X} + \left\| f_1 - A_1 \varphi \right\|^{2}_{Y, R^{-1}}. \label{4.1}
\end{equation}
By using a weighted norm $\left\| \cdot \right\|^{2}_{X, B_{0}^{-1}}:=\langle \cdot, B_{0}^{-1} \cdot \rangle_{X}$ where $B_{0}:= \frac{1}{\alpha} I$, the functional $J_1$ can be of the form
\begin{equation}
J_{1}(\varphi)= \left\| \varphi - \varphi_0 \right\|^{2}_{X, B^{-1}_{0}} + \left\| f_1 - A_1 \varphi \right\|^{2}_{Y, R^{-1}}, \label{4.2}
\end{equation}
and its unique minimizer $\varphi_1$ is given by
\begin{equation}
\varphi_{1} := \varphi_0 + (I + A_{1}^{*}A_{1})^{-1}A_{1}^{*}\left(f - A_1\varphi_0 \right), \label{4.3}
\end{equation}
where $A^{*}_{1}$ is the adjoint operator with respect to weighted scalar products $\langle \cdot, \cdot \rangle_{X, B_{0}^{-1}}$ and $\langle \cdot, \cdot \rangle_{Y, R^{-1}}$. 
We calculate
\begin{eqnarray}
\langle f, A_{1} \varphi \rangle_{Y, R^{-1}} &=& \langle f, R^{-1}A_{1} \varphi \rangle_{Y}
\nonumber\\
&=& \langle A^{H}_{1} R^{-1} f, \varphi \rangle_{X}
\nonumber\\
&=& \langle B_{0} A^{H}_{1} R^{-1} f, \varphi \rangle_{X, B_{0}^{-1}}, \label{4.4}
\end{eqnarray}
which implies that 
\begin{equation}
A^{*}_{1}=B_{0}A^{H}_{1}R^{-1}, \label{4.5}
\end{equation}
where $A_{1}^{H}$ is the adjoint operator with respect to usual scalar products $\langle \cdot, \cdot \rangle_{X}$ and $\langle \cdot, \cdot \rangle_{Y}$. 
Then, we have
\begin{eqnarray}
\varphi_{1} &=& \varphi_0 + (I + B_{0}A^{H}_{1}R^{-1}A_{1})^{-1}B_{0}A^{H}_{1}R^{-1}\left(f - A_1\varphi_0 \right)
\nonumber\\
&=&\varphi_0 +  (B^{-1}_{0} + A^{H}_{1}R^{-1}A_{1})^{-1}A^{H}_{1}R^{-1}\left(f_1 - A_1\varphi_0 \right). \label{4.6}
\end{eqnarray}
\par
Next, we assume that one more measurement $f_2 \in Y$ and observation operator $H_2$ are given.
The functional for two measurements is given by
\begin{eqnarray}
J_{Full, 2}(\varphi)&:=& \left\| \varphi - \varphi_0 \right\|^{2}_{X, B_{0}^{-1}} + \left\| f_1 - A_1 \varphi \right\|^{2}_{Y, R^{-1}} + \left\| f_2 - A_2\varphi \right\|^{2}_{Y, R^{-1}}. 
\nonumber
\\
&=& J_{1}(\varphi) + \left\| f_2 - A_2\varphi \right\|^{2}_{Y, R^{-1}}. 
\label{4.7}
\end{eqnarray}
The question is whether we can find $B_1$ such that $J_{Full, 2}(\varphi)=J_2(\varphi) + c$ where $c$ is a constant number independently of $\varphi$, and the functional $J_{2}(\varphi)$ is defined by
\begin{equation}
J_{2}(\varphi) = \left\| \varphi - \varphi_1 \right\|^{2}_{X, B_{1}} + \left\| f_2 - A_2 \varphi \right\|^{2}_{Y, R^{-1}},  \label{4.8}
\end{equation} 
where $\varphi_1$ is defined by (\ref{4.6}). 
To answer this question, we show the following lemma.
\begin{lem}\label{lemma 4.1}
Set $B_1:=\left(B^{-1}_{0} + A^{H}_{1}R^{-1}A_{1}\right)^{-1}$. Then, 
\begin{equation}
J_{1}(\varphi) = \left\| \varphi - \varphi_1 \right\|^{2}_{X, B_{1}^{-1}} + c, \label{4.9}
\end{equation}
where $c$ is some constant independently of $\varphi$. 
\end{lem}
\begin{proof}
We calculate
\begin{eqnarray}
J_{1}(\varphi)&=&\left\langle \varphi - \varphi_0, B^{-1}_{0}\left(\varphi - \varphi_0 \right) \right\rangle_{X} + \left\langle f_1 - A_{1} \varphi, R^{-1}\left( f_1 - A_{1} \varphi \right) \right\rangle_{Y}
\nonumber\\
&=& \left\langle \varphi, B^{-1}_{0}\varphi \right\rangle_{X} -2 \mathrm{Re} \left\langle \varphi, B^{-1}_{0}\varphi_0 \right\rangle_{X} + \left\langle \varphi_0, B^{-1}_{0}\varphi_0 \right\rangle_{X}
\nonumber\\
&&\ \ \ + \left\langle f_1, R^{-1}f_1 \right\rangle_{Y} -2 \mathrm{Re}\left\langle \varphi, A^{H}_{1}R^{-1}f_1 \right\rangle_{X} + \left\langle \varphi, A^{H}_{1}R^{-1}A_{1}\varphi \right\rangle_{X}.
\nonumber\\
&=& \left\langle \varphi, B^{-1}_{0}\varphi \right\rangle_{X} -2 \mathrm{Re} \left\langle \varphi, B^{-1}_{0}\varphi_0 \right\rangle_{X} -2 \mathrm{Re}\left\langle \varphi, A^{H}_{1}R^{-1}f_1 \right\rangle_{X} 
\nonumber\\
&&\ \ \ \ \ + \left\langle \varphi, A^{H}_{1}R^{-1}A_{1}\varphi \right\rangle_{X} + c_0
\nonumber\\
&=& \left\langle \varphi, B^{-1}_{1}\varphi \right\rangle_{X} -2 \mathrm{Re} \left\langle \varphi, B^{-1}_{0}\varphi_0 \right\rangle_{X} -2 \mathrm{Re}\left\langle \varphi, A^{H}_{1}R^{-1}f_1 \right\rangle_{X} + c_0, \nonumber\\ \label{4.10}
\end{eqnarray}
where we used $B_{1}^{-1}=\left(B^{-1}_{0} + A^{H}_{1}R^{-1}A_{1}\right)$. 
By (\ref{4.6}), we have 
\begin{eqnarray}
B^{-1}_{1}\left(\varphi - \varphi_{1} \right) &=& B^{-1}_{1}\varphi - B^{-1}_{1} \varphi_{1} \nonumber\\
&=&B^{-1}_{1}\varphi - \left(B^{-1}_{0} + A^{H}_{1}R^{-1}A_{1}\right) \varphi_{0} - A^{H}_{1}R^{-1}\left(f - A_1\varphi_0 \right) \nonumber\\
&=&B^{-1}_{1}\varphi - B^{-1}_{0}\varphi_0 - A^{H}_{1}R^{-1}f_1. \label{4.11}
\end{eqnarray}
By using (\ref{4.11}) and the self-adjointness of $B^{-1}_{1}$, we have
\begin{eqnarray}
&&\left\langle \varphi - \varphi_{1}, B^{-1}_{1}\left(\varphi - \varphi_{1} \right) \right\rangle_{X}
\nonumber\\
&=&\left\langle \varphi - \varphi_1, B^{-1}_{1}\varphi - B^{-1}_{0}\varphi_0 - A^{H}_{1}R^{-1}f_1 \right\rangle_{X} 
\nonumber\\
&=& \left\langle B^{-1}_{1} \left(\varphi - \varphi_1\right), \varphi  \right\rangle_{X} -  \left\langle \varphi, B^{-1}_{0}\varphi_0 \right\rangle_{X} -  \left\langle \varphi,  A^{H}_{1}R^{-1}f \right\rangle_{X} + c_1
\nonumber\\
&=&\left\langle B^{-1}_{1}\varphi - B^{-1}_{0}\varphi_0 - A^{H}_{1}R^{-1}f,  \varphi \right\rangle_{X}
\nonumber\\
&&\ \ \ \ \ \ \ \ \ \ \ \ \ -  \left\langle \varphi, B^{-1}_{0}\varphi_0 \right\rangle_{X} -  \left\langle \varphi,  A^{H}_{1}R^{-1}f \right\rangle_{X} + c_1
\nonumber\\
&=&\left\langle \varphi, B^{-1}_{1}\varphi \right\rangle_{X} -2 \mathrm{Re} \left\langle \varphi, B^{-1}_{0}\varphi_0 \right\rangle_{X} -2 \mathrm{Re}\left\langle \varphi, A^{H}_{1}R^{-1}f_1 \right\rangle_{X}  + c_1.
\nonumber\\ \label{4.12}
\end{eqnarray}
With (\ref{4.10}) and (\ref{4.12}), $J_{1}(\varphi)$ is of the form
\begin{equation}
J_{1}(\varphi)=\left\langle \varphi - \varphi_{1}, B^{-1}_{1}\left(\varphi - \varphi_{1} \right) \right\rangle_{X} + c_2. \label{4.13}
\end{equation}
where $c_0$, $c_1$, and $c_2$ are some constant numbers independently of $\varphi$. 
Lemma \ref{lemma 4.1} has been shown.
\end{proof}
This lemma tells us that $J_{Full, 2}(\varphi)$ is equivalent to $J_{2}(\varphi)$ in the sense of minimization with respect to $\varphi$. 
By the same argument in (\ref{4.2})--(\ref{4.6}), its unique minimizer $\varphi_2$ is given by 
\begin{equation}
\varphi_{2} := \varphi_1 +  (B^{-1}_{1} + A^{H}_{2}R^{-1}A_{2})^{-1}A^{H}_{2}R^{-1}\left(f_2 - A_2\varphi_1 \right). \label{4.14}
\end{equation}
\par
We can repeat the above arguments (\ref{4.1})--(\ref{4.14}) until all measurements $f_1,...,f_n$ and all observation operators $A_1,...,A_n$ are given. 
Then, we have following algorithms
\begin{equation}
\varphi_{n} := \varphi_{n-1} + K_{n}\left( f_{n}-A_{n} \varphi_{n-1} \right), \label{4.15}
\end{equation}
where the operator 
\begin{equation}
K_{n}:= \left( B^{-1}_{n-1} + A^{H}_{n}R^{-1}A_{n}\right)^{-1} A^{H}_{n}R^{-1}, \label{4.16}
\end{equation}
is called the {\it Kalman gain matrix}, and $B_n$ is defined by
\begin{equation}
B_n := \left(B^{-1}_{n-1} + A^{H}_{n}R^{-1}A_{n}\right)^{-1}. \label{4.17}
\end{equation}
Since we have
\begin{eqnarray}
\left( B^{-1}_{n-1} + A^{H}_{n}R^{-1}A_{n}\right)B_{n-1}A^{H}_{n}
&=&
A^{H}_{n} + A^{H}_{n}R^{-1}A_{n}B_{n-1}A^{H}_{n}
\nonumber\\
&=&
A^{H}_{n}R^{-1}\left(R + A_{n} B_{n-1} A^{H}_{n} \right), \nonumber
\end{eqnarray}
the Kalman gain matrix $K_{n}$ can be of the form
\begin{equation}
K_{n}= B_{n-1} A^{H}_{n}\left(R + A_{n} B_{n-1} A^{H}_{n} \right)^{-1}. \nonumber
\end{equation}
Here, we show the following lemma that the operator $B_n$ has another form.
\begin{lem}
Let $K_n$ be the Kalman gain matrix defined in (\ref{4.16}). 
Then, the operator $B_n$ has the following form
\begin{equation}
B_n = \left( I - K_n A_n \right) B_{n-1}. \label{4.18}
\end{equation}
\end{lem}
\begin{proof}
By multiplying (\ref{4.16}) by $\left( B^{-1}_{n-1} + A^{H}_{n}R^{-1}A_{n} \right)$ from the left hand side, and by $A_n$ from right hand side, we have
\begin{equation}
\left( B^{-1}_{n-1} + A^{H}_{n}R^{-1}A_{n} \right)K_{n}A_{n} = A^{H}_{n}R^{-1} A_{n}, \label{4.19}
\end{equation}
which implies that by using (\ref{4.17})
\begin{eqnarray}
B^{-1}_{n}\left(I-K_{n}A_{n}\right)&=& \left( B^{-1}_{n-1} + A^{H}_{n}R^{-1}A_{n}\right)\left(I-K_{n}A_{n}\right)
\nonumber\\
&=&\left( B^{-1}_{n-1} + A^{H}_{n}R^{-1}A_{n}\right) - A^{H}_{n}R^{-1} A_{n}
\nonumber\\
&=& B^{-1}_{n-1}. \label{4.20}
\end{eqnarray}
Multiplying (\ref{4.20}) by $B_{n}$ from the left hand side, and by $B_{n-1}$ from the right hand side, we finally get (\ref{4.18}). 
\end{proof}
We summarize the update formula in the following.
\begin{equation}
\varphi^{KF}_{n}:= \varphi^{KF}_{n-1} + K_{n}\left( f_{n}-A_{n} \varphi^{KF}_{n-1} \right), \label{4.21}
\end{equation}
\begin{equation}
K_{n}:= B_{n-1} A^{H}_{n}\left(R + A_{n} B_{n-1} A^{H}_{n} \right)^{-1}, \label{4.22}
\end{equation}
\begin{equation}
B_{n}:= \left(I - K_{n} A^{H}_{n} \right)B_{n-1}, \label{4.23}
\end{equation}
for $n=1,...,N$, where $\varphi^{KF}_{0}:=\varphi_0$ and $B_{0}:=\frac{1}{\alpha}I$. We call this the {\it Kalman filter}.
\par
We observe the above algorithm. 
It means that we can estimate the state $\varphi$ every time $n$ to observe one measurement $f_{n}$ without waiting all measurements $\{ f_{n} \}_{n=1}^{N}$. 
It includes not only the update (\ref{4.21}) of the state $\varphi$, but also the update (\ref{4.23}) of the weight $B$ of the norm, which plays the role of keeping the information of the previous state. 
In finite dimensional setting, the weight $B$ is also interpreted as the covariance matrices of the state error distribution from statistical viewpoint (see Section \ref{Stochastic viewpoints of Kalman filter}).
\par
Finally in this section, we show the equivalence of Full data Tikhonov and Kalman filter when all observation operators $A_n$ are linear.
\begin{thm} \label{equivalence}
For measurements $f_1,...,f_N$, linear operators $A_1,...,A_N$, and the initial guess $\varphi_0 \in X$, the final sate of the Kalman filter given by (\ref{4.21})--(\ref{4.23}) is equivalent to the state of the Full data Tikhonov given by (\ref{3.8}), that is
\begin{equation}
\varphi^{KF}_{N}=\varphi^{FT}_{N}. \label{4.24}
\end{equation}
\end{thm}
\begin{proof}
It is sufficient to show that
\begin{equation}
J_{Full, N}(\varphi) = \left\| \varphi - \varphi^{KF}_{N} \right\|^{2}_{X, B_{N}^{-1}} + c_N, \label{4.25}
\end{equation}
where $c_N$ is some constant independently of $\varphi$. 
We will prove (\ref{4.25}) by the induction. 
The case of $N=1$ has already been shown in Lemma 4.1. 
\par
We assume that (\ref{4.25}) in the case of $n \in \mathbb{N}$ with $1\leq n\leq N-1$ holds, that is, 
\begin{equation}
J_{Full,n}(\varphi) = \left\| \varphi - \varphi^{KF}_{n} \right\|^{2}_{X, B_{n}^{-1}} + c_{n}, \label{4.26}
\end{equation}
where $c_{n}$ is some constant. 
Then, we have
\begin{eqnarray}
J_{Full, n+1}(\varphi) &=& J_{Full, n}(\varphi) + \left\| f_{n+1} - A_{n+1} \varphi \right\|^{2}_{Y, R^{-1}} 
\nonumber\\
&=&  \left\| \varphi - \varphi^{KF}_{n} \right\|^{2}_{X, B_{n}^{-1}} + \left\| f_{n+1} - A_{n+1} \varphi \right\|^{2}_{Y, R^{-1}} + c_{n}. \label{4.27}
\end{eqnarray}
By the same argument in Lemma \ref{lemma 4.1} replacing $B_0$, $\varphi_0$, $f_1$, $A_1$ by $B_{n}$, $\varphi_{n}$, $f_{n+1}$, $A_{n+1}$, respectively, we have that $J_{Full, n+1}(\varphi)=\left\| \varphi - \varphi^{KF}_{n+1} \right\|^{2}_{X, B_{n+1}^{-1}} + c_{n+1}$. 
Theorem \ref{equivalence} has been shown. 
\end{proof}
\begin{rem}\label{Remark4.2}
If $\vec{f}$ is true measurement, i.e., $\vec{A}\varphi_{true}=\vec{f}$ and $\vec{A}$ is injective, then our Kalman filter solution $\varphi_{N}^{KF}=\varphi_{N}^{FT}$, which is is equal to the Full data Tikhonov solution $\varphi_{N}^{FT}$, convergences to true $\varphi_{true}$ as $\alpha \to 0$ (see (iv) in Lemma \ref{lemma Tikhonov}).
The injectivity of $\vec{A}$ would be expected when the number $N$ of measurement is large enough.
\end{rem}
\section{Stochastic viewpoints of Kalman filter}\label{Stochastic viewpoints of Kalman filter}
In this section, we observe the Kalman filter (\ref{4.21})-(\ref{4.23}) from Bayesian viewpoints. 
For simplicity, we assume that $X=\mathbb{C}^{m}$ and $Y=\mathbb{C}^{l}$, $m, l \in \mathbb{N}$, and we treat the state $\bm{\varphi} \in \mathbb{C}^{m}$ and the measurement $\bm{f}  \in \mathbb{C}^{l}$ as complex random vectors.
We recall that Bayes' theorem 
\begin{equation}
p(\bm{\varphi}|\bm{f})\propto p(\bm{f}|\bm{\varphi})p(\bm{\varphi})
\end{equation}
where $p(\bm{\varphi})$ is the prior probability of the state $\bm{\varphi}$ before the measurement $f$ which is modeled by information of the current state $\bm{\varphi}$, $p(\bm{f}|\bm{\varphi})$ is the probability of observing $f$ given $\bm{\varphi}$ which is called the likelihood, and $p(\bm{\varphi}|\bm{f})$ is the posterior probability of the state $\bm{\varphi}$ given the measurement $\bm{f}$. 
Bayesian theory is a simple and generic approach which can be applied to inverse and ill-posed problems (see e.g., \cite{ arridge2019solving, calvetti2007introduction, Freitag,  NakamuraPotthast, stuart2010inverse}). 
\par
Here, we recall that complex Gaussian distribution (see e.g., \cite{gallager2008circularly, picinbono1996second, van1995multivariate}). 
Let us first remind that a complex random variable $\bm{z}$ of $\mathbb{C}^{n}$ is a pair
of real random variable of $\mathbb{R}^{n}$ such that $\bm{z} = \bm{x} + i\bm{y}$. 
A complex random variable $\bm{z}$ is said to be Gaussian if its real and imaginary parts $\bm{x}$ and $\bm{y}$ are jointly Gaussian. 
Its distribution with zero mean is 
\begin{equation}
p(\bm{z})=p(\bm{x}, \bm{y})= \frac{1}{\sqrt{(2\pi)^{2n} |\Sigma_{2n}|}} e^{-\frac{1}{2} \bm{v}^{T} \Sigma^{-1}_{2n} \bm{v} }, \label{Complex Gaussian distribution}
\end{equation}
where $\bm{v} \in \mathbb{R}^{2n}$ such that  $\bm{v^{T}}=(\bm{x^{T}}, \bm{y^{T}})$, and $\bm{T}$ means transposition, and $\bm{\Sigma_{2n}} \in \mathbb{R}^{2n \times 2n}$ is the covariance matrix defined by
\begin{equation}
\bm{\Sigma_{2n}} = \left(
    \begin{array}{cc}
    E(\bm{x}\bm{x^{T}}) & E(\bm{x}\bm{y^{T}}) \\
    E(\bm{y}\bm{x^{T}}) & E(\bm{y}\bm{y^{T}})
    \end{array}
  \right).
\end{equation}
If real and imaginary parts are independent Gaussian distributed random variables with mean zero and same covariance $\bm{\Sigma} \in \mathbb{R}^{n \times n}$, then (\ref{Complex Gaussian distribution}) can be computed as
\begin{equation}
p(\bm{z})= \frac{1}{(\pi)^{n} |\bm{\Sigma}|} e^{-\bm{z^{H}} \bm{\Sigma}^{-1} \bm{z} }.
\end{equation}
where $\bm{H}$ means transposition and complex conjugation. 
This distribution is referred to as {\it circularly-symmetric (central) complex Gaussian distribution}, and is denoted by $\mathcal{C}\mathcal{N}(0, \Sigma)$. 
\par
We assume that complex vector $\bm{\varphi^{KF}_{n-1}} \in \mathbb{C}^{m}$ and the positive definite matrix $\bm{B_{n-1}} \in \mathbb{R}^{m \times m}$ are determined in some way, and the prior $p(\bm{\varphi})$ is modeled by a circularly-symmetric complex Gaussian distribution $\mathcal{C}\mathcal{N}(\bm{\varphi^{KF}_{n-1}}, \bm{B_{n-1}})$, that is, 
\begin{equation}
p(\bm{\varphi})= \frac{1}{(\pi)^{m} |\bm{B_{n-1}}|} e^{-(\bm{\varphi} - \bm{\varphi^{KF}_{n-1}})^{\bm{H}} \bm{B_{n-1}}^{-1} (\bm{\varphi} - \bm{\varphi^{KF}_{n-1}}) }.
\end{equation}
Furthermore, we assume that the observation error $\bm{f} - \bm{A}\bm{\varphi}$ is distributed from $\mathcal{C}\mathcal{N}(0, \bm{R})$ where the $\bm{R} \in \mathbb{R}^{l \times l}$ is some positive definite matrix. 
Then, the likelihood $p(\bm{f}=\bm{f_{n}}|\bm{\varphi})$ is modeled by
\begin{equation}
p(\bm{f}=\bm{f_{n}}|\bm{\varphi}) = \frac{1}{(\pi)^{l} |R|} e^{-(\bm{f_{n}} - \bm{A_{n}} \bm{\varphi})^{\bm{H}} R^{-1} (\bm{f_{n}} - \bm{A_{n}} \bm{\varphi}) }
\end{equation}
Then by Bayes' theorem, the posterior $p(\varphi|\bm{f}=\bm{f_{n}})$ can be computed as
\begin{eqnarray}
p(\bm{\varphi}|\bm{f}=\bm{f_{n}}) &\propto& e^{ -\{ (\bm{\varphi} - \bm{\varphi^{KF}_{n-1}})^{\bm{H}} \bm{B_{n-1}}^{-1} (\bm{\varphi} - \bm{\varphi^{KF}_{n-1}}) + (\bm{f_{n}} - \bm{A_{n}} \bm{\varphi})^{\bm{H}} \bm{R}^{-1} (\bm{f_{n}} - \bm{A_{n}} \bm{\varphi}) \} }
\nonumber\\
&\propto& e^{- (\bm{\varphi} - \bm{\varphi^{KF}_{n}})^{H}\bm{B_{n-1}}^{-1}(\bm{\varphi} - \bm{\varphi^{KF}_{n}})}
\end{eqnarray}
where $\bm{\varphi^{KF}_{n}}$ and $\bm{B_{n}}$ is defined by (\ref{4.21}) and (\ref{4.23}), respectively.
This computation is guaranteed by the same argument in Section \ref{Kalman Filter Section}.
Therefore, posterior distribution is a circularly-symmetric complex Gaussian distribution $\mathcal{C}\mathcal{N}(\bm{\varphi^{KF}_{n}}, \bm{B_{n}})$ with mean $\bm{\varphi^{KF}_{n}} \in \mathbb{C}^{m}$ and covariance matrix $\bm{B_{n}} \in \mathbb{R}^{m \times m}$, which means that the Kalman filter update in (\ref{4.21})-(\ref{4.23}) can be interpreted as updating mean and covariance matrix of Gaussian distribution of the state in the case that the prior and likelihood are assumed to be Gaussian.
\section{Numerical examples}\label{Numerical examples}
In this section, we provide numerical examples for the Kalman filter algorithm. 
Our inverse scattering problem is to solve the linear integral equation
\begin{equation}
\mathcal{F}_{B, n}q=u^{\infty}_{B}(\cdot, \theta_n), \label{5.1}
\end{equation}
for $n=1,...,N$ where the operator $\mathcal{F}_{B, n}:L^{2}(Q) \to L^{\infty}(\mathbb{S}^{1})$ is defined by 
\begin{equation}
\mathcal{F}_{B, n}q(\hat{x}):=\mathcal{F}_{B}q(\hat{x},\theta_n)=\frac{k^2}{4\pi}\int_{Q}\mathrm{e}^{ik (\theta_{n} - \hat{x}) \cdot y} q(y)dy, \label{5.2}
\end{equation}
where the incident direction is given by $\theta_n:=\left(\mathrm{cos}(2\pi n/N), \mathrm{sin}(2\pi n/N) \right)$ for each $n=1,...,N$.
We assume that the support $Q$ of the function $q$ is included in the square $[-S, S]^2$ with some $S>0$. 
\par
The linear integral equation (\ref{5.1}) is discretized by
\begin{equation}
\bm{\mathcal{F}_{n}} \bm{q} = \bm{u^{ \infty}_{n}},
\end{equation}
where
\begin{equation}
\bm{\mathcal{F}_{n}} = \frac{k^2 S^{2}}{4\pi M^{2}} \left(
\mathrm{e}^{ik( \theta_{n}-\hat{x}_j) \cdot y_{i,l}} \right)_{j=1,...,J, \ -M \leq i,l \leq M-1} \in \mathbb{C}^{J\times (2M)^2}. \label{5.5}
\end{equation} 
where $y_{i,l}:=\left( \frac{(2i+1)S}{2M}, \frac{(2l+1)S}{2M} \right)$, and $M \in \mathbb{N}$ is a number of the division of $[0,S]$ (i.e., the function $q$ is discretized by piecewise constant on $[-S, S]^{2}$ which is decomposed by squares with the length $\frac{S}{M}$), and $\hat{x}_j:=\left(\mathrm{cos}(2\pi j/J),  \mathrm{sin}(2\pi j/J) \right)$, and $J \in \mathbb{N}$ is a number of the division of $[0,2\pi]$ and 
\begin{equation}
\bm{q} = \left( q(y_{i,l})  \right)_{-M \leq i,l \leq M-1} \in \mathbb{C}^{(2M)^2}, \label{5.4}
\end{equation} 
and
\begin{equation}
\bm{u^{ \infty}_{n}} = \left( u^{\infty}_{B}(\hat{x}_j, \theta_{n}) \right)_{j=1,...,J} +\bm{\epsilon_{n}} \in \mathbb{C}^{J}. \label{5.3}
\end{equation}
The noise $\bm{\epsilon_{n}} \in \mathbb{C}^{J}$ is sampling from a complex Gaussian distribution $\mathcal{C}\mathcal{N}(0, \sigma^{2}\bm{I})$, which is equivalent to $\bm{\epsilon_{n}}=\bm{\epsilon^{re}_{n}}+i\bm{\epsilon^{im}_{n}}$ where $\bm{\epsilon^{re}_{n}}, \bm{\epsilon^{im}_{n}} \in \mathbb{R}^{J}$ are independently identically distributed from Gaussian distribution $\mathcal{N}(0, \sigma^{2}\bm{I})$ with mean zero and covariance matrix $\sigma^{2}\bm{I}$ where $\sigma>0$.
\par
Here, we always fix discretization parameters as $J=30$, $M=8$, $S=3$, and weight $\bm{R} \in \mathbb{R}^{J \times J}$, which is the covariance matrix of the observation error distribution, as $R=r^{2}I$, and $r=1$. 
From Remarks \ref{Remark4.2} and \ref{Remark for Tikhonov}, in order to converge to true solution, the matrix $\vec{\bm{\mathcal{F}}}:=\left(
\begin{array}{cc}
\bm{\mathcal{F}_{1}} \\
\vdots \\
\bm{\mathcal{F}_{N}}
\end{array}
\right) \in \mathbb{C}^{NJ \times (2M)^{2}}$ should be injective. The necessary condition is $JN>(2M)^{2}$, so we choose the parameter $N=30$ ($NJ=30\times30=900>$ $(2M)^2=256$).
\par
We consider true functions as the characteristic function
\begin{equation}
q^{true}_{j}(x):=\left\{ \begin{array}{ll}
1 & \quad \mbox{for $x \in B_j$}  \\
0 & \quad \mbox{for $x \notin B_j$}
\end{array} \right., \label{5.6}
\end{equation}
where the support $B_j$ of the true function is considered as the following two types.
\begin{equation}
B_1:=\left\{(x_1, x_2) : x^{2}_{1}+x^{2}_{2} <1.5   \right\}, \label{5.7}
\end{equation}
\begin{equation}
B_2:=\left\{(x_1, x_2):\begin{array}{cc}
      (x_{1}+1.5)^2+(x_{2}+1.5)^{2} < (1.0)^{2}\ or \\
      1 < x_1 < 2,\ -2 < x_2 < 2\ or \\
      -2 < x_1 < 2,\ -2.0 < x_2 < -1.0 
    \end{array}
\right\}. \label{5.8}
\end{equation}
In Figure \ref{truefuctions}, the blue closed curve is the boundary $\partial B_j$ of the support $B_j$, and the green brightness indicates the value of the true function on each cell divided into $(2M)^2=256$ in the sampling domain $[-S, S]^2=[-3, 3]^2$.
Here, we always employ the initial guess $q_0$ as
\begin{equation}
q_0\equiv0. \label{5.9}
\end{equation}
\par
Figure \ref{comparisonKFFT} shows the reconstruction by the Kalman filter (KF) and the Full data Tikhonov (FT) discussed in (\ref{4.21})--(\ref{4.23}) and (\ref{3.8}), respectively.
The first and second column correspond to visualization of the updated state $q$ in the case when four measurements $\{u^{\infty}_{B}(\cdot, \theta_n)\}_{n=1}^{4}$ and twenty measurements $\{u^{\infty}_{B}(\cdot, \theta_n)\}_{n=1}^{20}$ are given, respectively, for different methods KF and FT, and for two different shapes $B_1$ and $B_2$. 
In Figure \ref{comparisonKFFT}, the wave number and the regularization parameter are fixed as $k=3$ and $\alpha=1$, respectively, and the measurements are noisy free.
The third column corresponds to the graph of the Mean Square Error (MSE) defined by
\begin{equation}
e_{n}:=\left\|\bm{q^{true}}-\bm{q_{n}} \right\|^2, \label{5.10}
\end{equation}
where $\bm{q_{n}}$ is associated with the updated state given $n$ measurements. 
The horizontal axis is with respect to number of given measurements, and the vertical axis is the value of MSE. 
We observe that in Figure \ref{comparisonKFFT}, KF and FT are equivalent, which coincides with the theoretical result in Theorem \ref{equivalence}.
\par
Figures \ref{KFreconstruction} and \ref{KFreconstructionnosiy} show the reconstruction by the Kalman filter (KF) with $\sigma=0.1, 0.5$, respectively, for two different wave numbers $k=5$ and $k=1$ and two different shape $B_1$ and $B_2$. 
The first and second columns correspond to visualization of the final state given full measurements ($n=30$) for different regularization parameters $\alpha=10$ and $1e-1$, respectively. 
The third column corresponds to graphs of MSE, which have three evaluations with respect to $\alpha=10, 1, 1e-1$. 
The case of $k=1.0$ fails to reconstruct even with small noise (see Figures \ref{KFreconstruction}), that is, the state does not converge to zero even with increasing the number of measurements and decreasing regularization parameters.
This ill-posedness is because the rank of the full far field mapping $\vec{\bm{\mathcal{F}}}=\left(
\begin{array}{cc}
\bm{\mathcal{F}_{1}} \\
\vdots \\
\bm{\mathcal{F}_{N}}
\end{array}
\right) \in \mathbb{C}^{NJ \times (2M)^{2}}$ ($NJ=30\times30=900$, $(2M)^2=256$) degenerates when the wave number $k$ decreases.
Figure \ref{rank} shows its degeneracy.
The horizontal axis is with respect to wave numbers, and the vertical axis is the number of the rank of full far field mappings $\vec{\bm{\mathcal{F}}}$ (Maximum of rank is $256$).
\section*{Acknowledgments}
This work of the first author was supported by Grant-in-Aid for JSPS Fellows (No.21J00119), Japan Society for the Promotion of Science.

\bibliographystyle{plain}
\bibliography{KF.bib}

\begin{figure}[h]
  \begin{minipage}[b]{0.5\linewidth}
  \centering
  \includegraphics[keepaspectratio, scale=0.5]{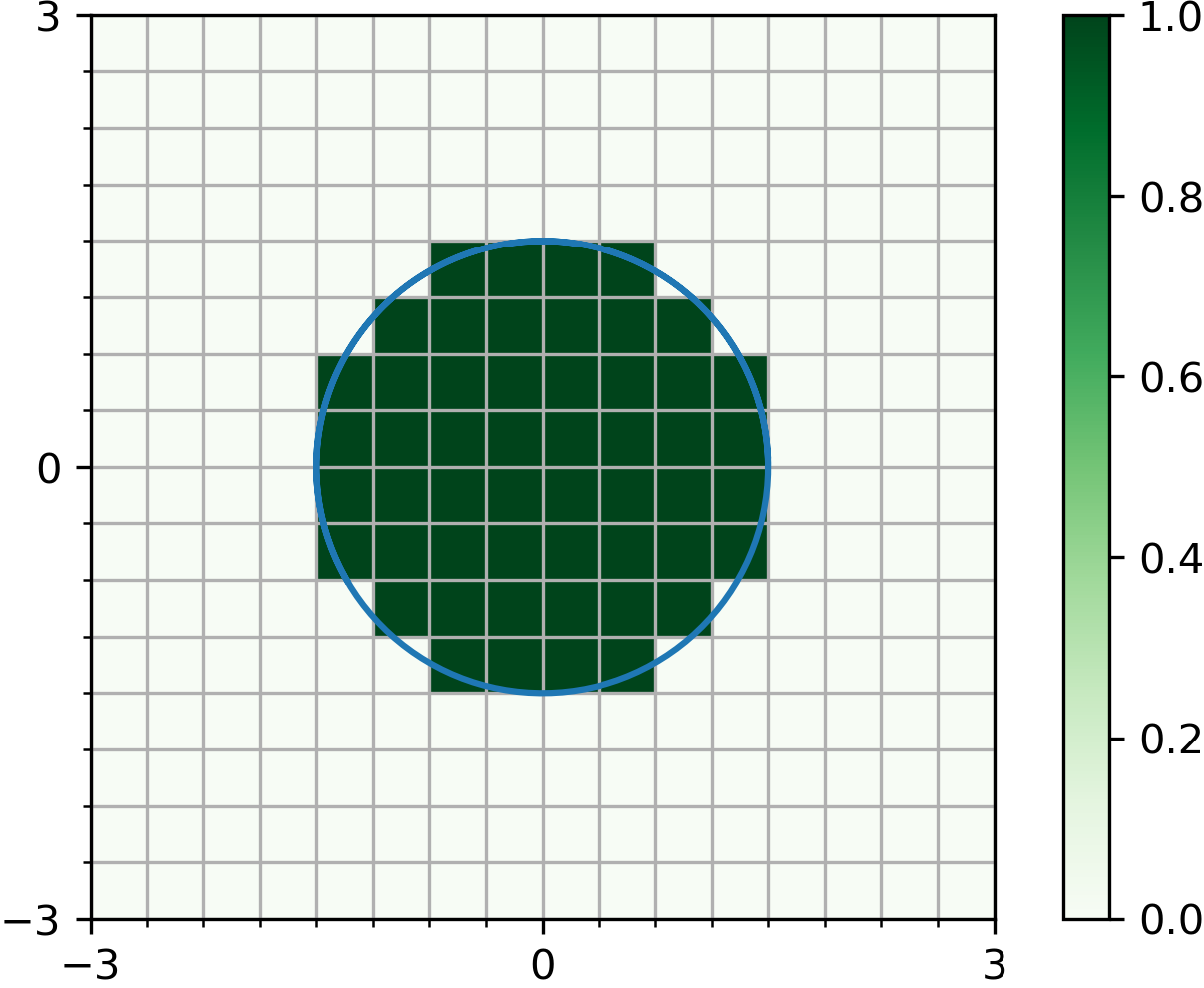}
  \subcaption{$q^{true}_{1}$}
 \end{minipage}
 \begin{minipage}[b]{0.5\linewidth}
  \centering
  \includegraphics[keepaspectratio, scale=0.5]
  {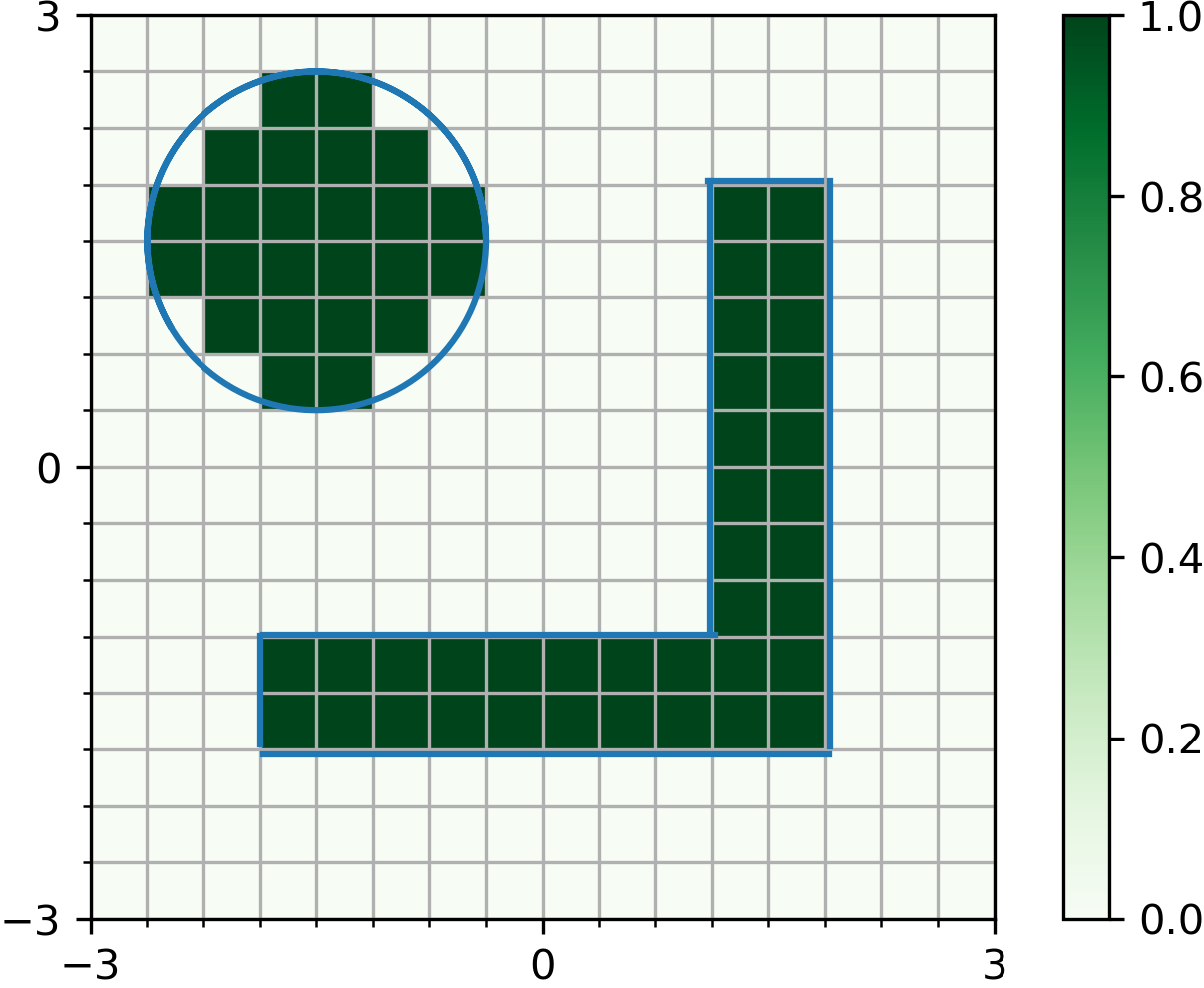}
  \subcaption{$q^{true}_{2}$}
 \end{minipage}
 
 \caption{true functions}
\label{truefuctions}
\end{figure}
\begin{figure}[h]
\begin{tabular}{c}
\hspace{-2.5cm}
 \begin{minipage}[b]{0.4\linewidth}
  \centering
  \includegraphics[keepaspectratio, scale=0.45]
  {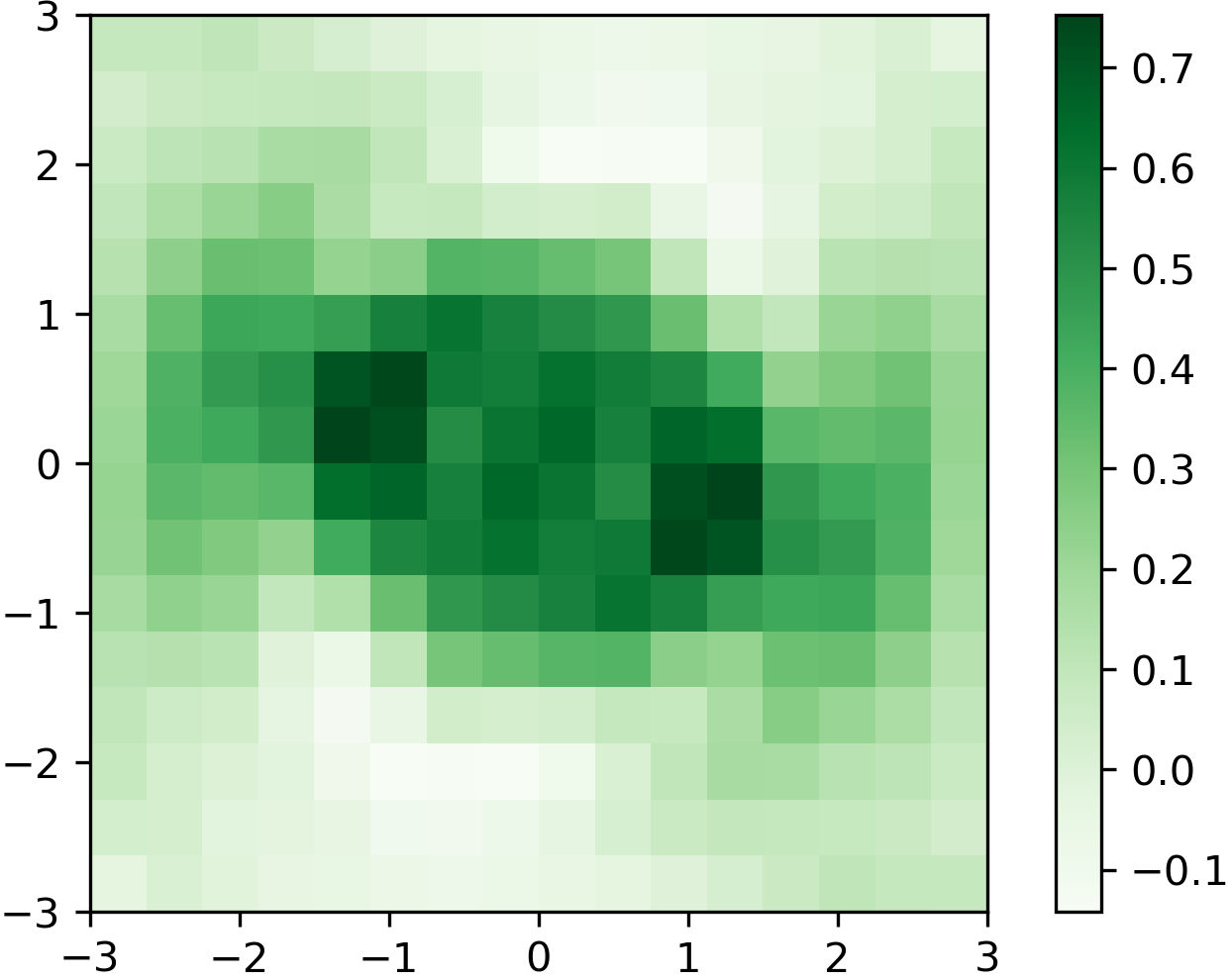}
  \subcaption{KF, $B_1$, $n=4$}
 \end{minipage}
 \begin{minipage}[b]{0.4\linewidth}
  \centering
  \includegraphics[keepaspectratio, scale=0.45]
  {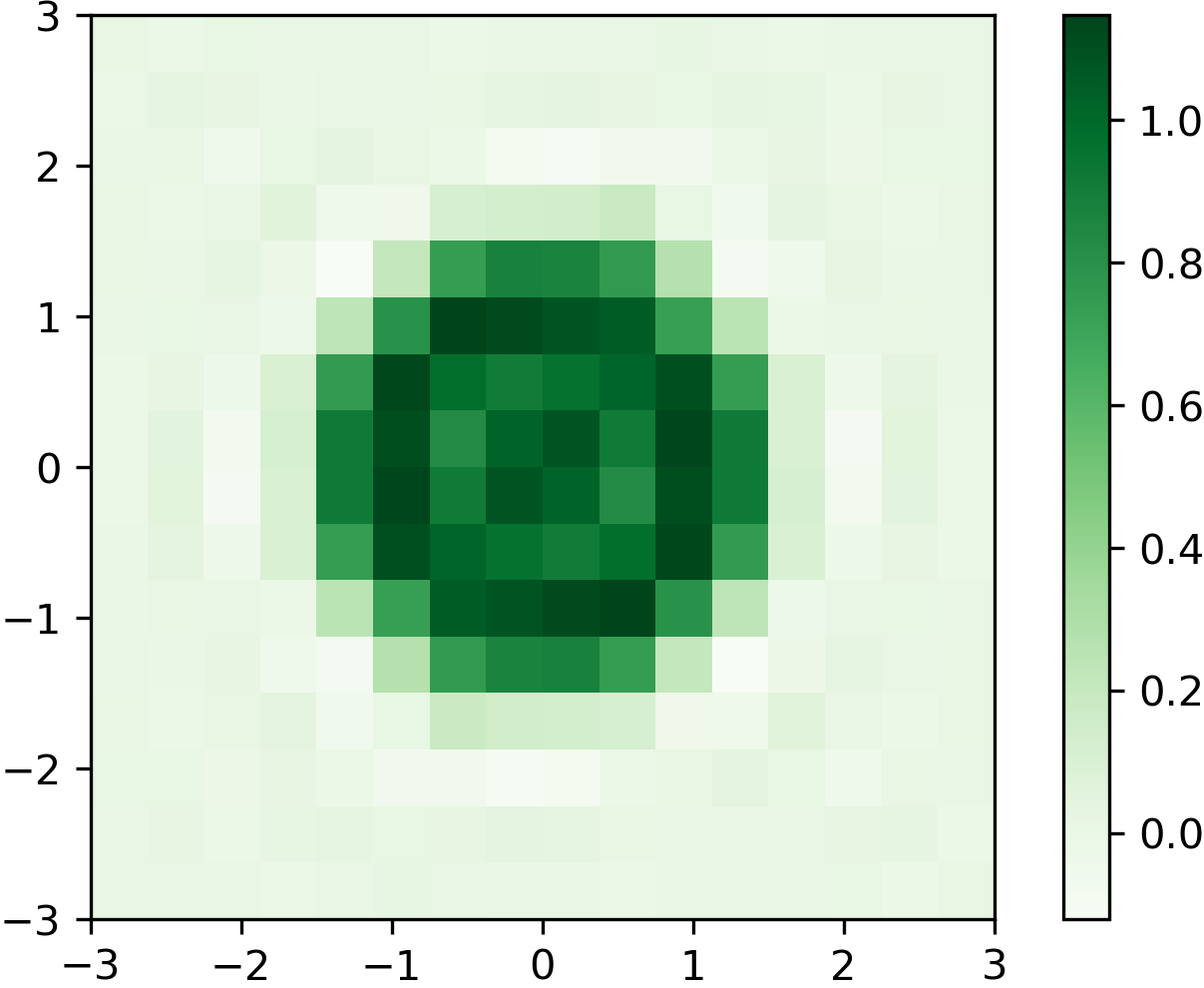}
  \subcaption{KF, $B_1$, $n=20$}
 \end{minipage}
 \begin{minipage}[b]{0.4\linewidth}
  \centering
  \includegraphics[keepaspectratio, scale=0.45]
  {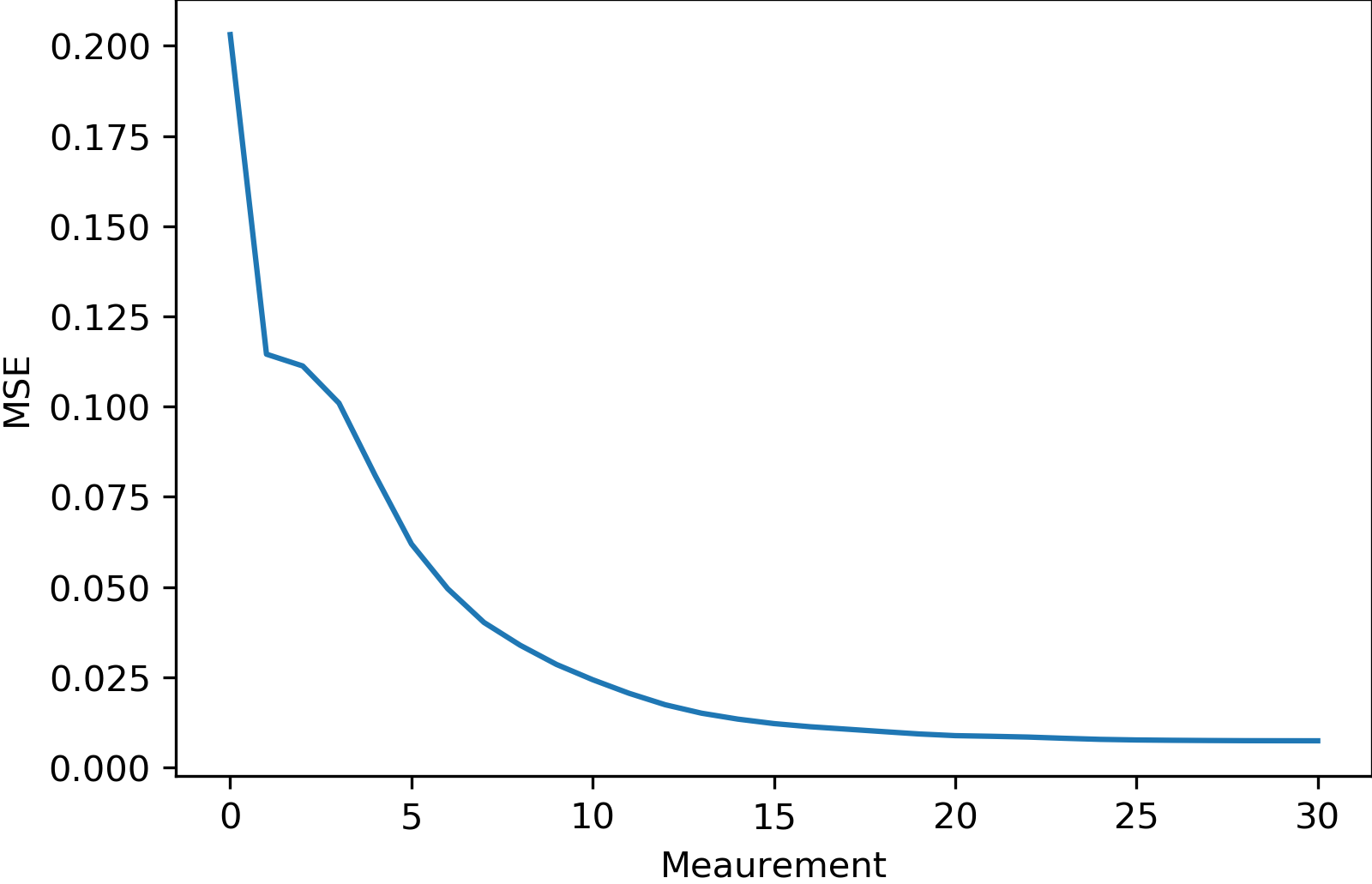}
  \subcaption{KF, $B_1$, error graph}
 \end{minipage}
\end{tabular}

\begin{tabular}{c}
\hspace{-2.5cm}
 \begin{minipage}[b]{0.4\linewidth}
  \centering
  \includegraphics[keepaspectratio, scale=0.45]
  {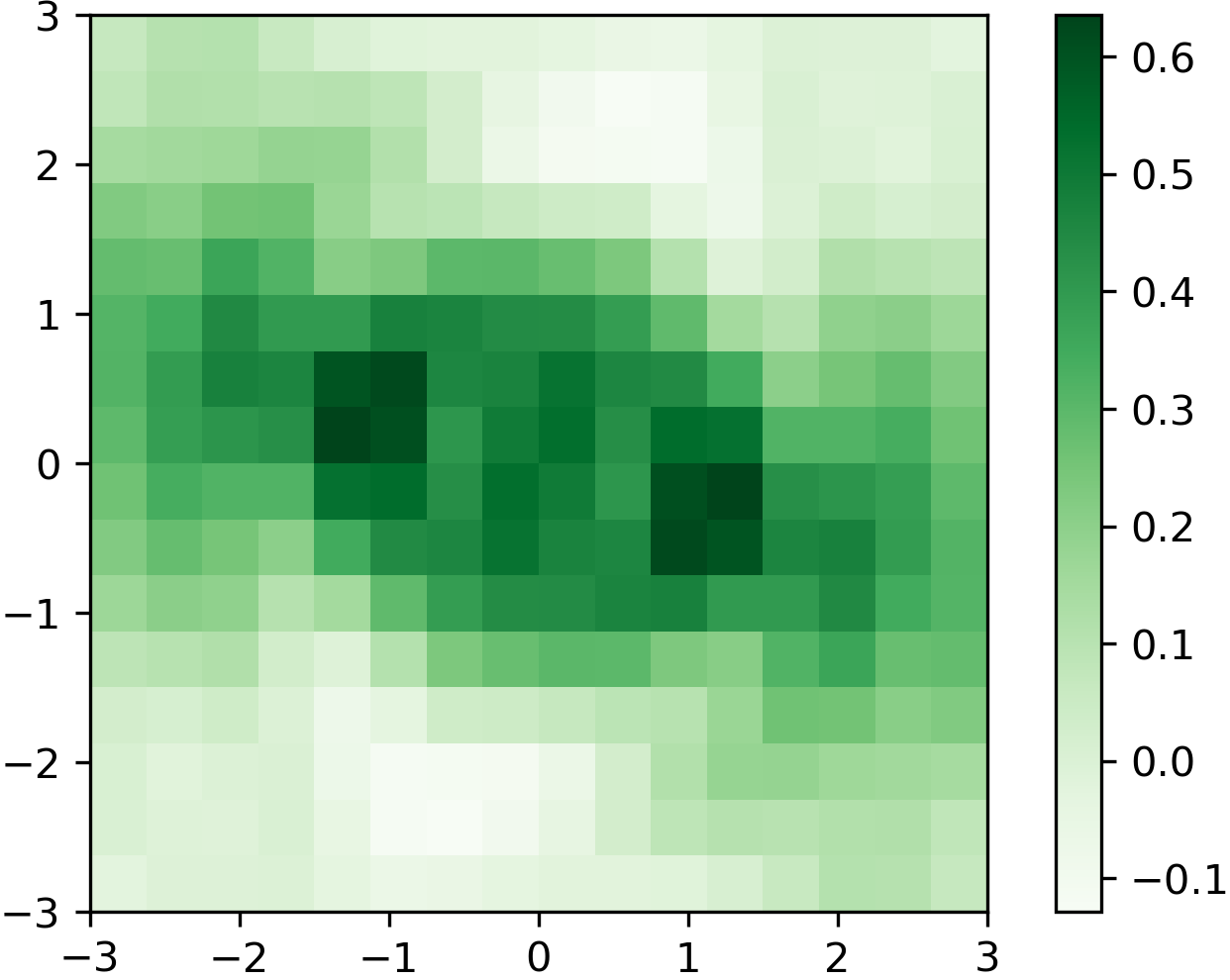}
  \subcaption{FT, $B_1$, $n=4$}
 \end{minipage}
 \begin{minipage}[b]{0.4\linewidth}
  \centering
  \includegraphics[keepaspectratio, scale=0.45]
  {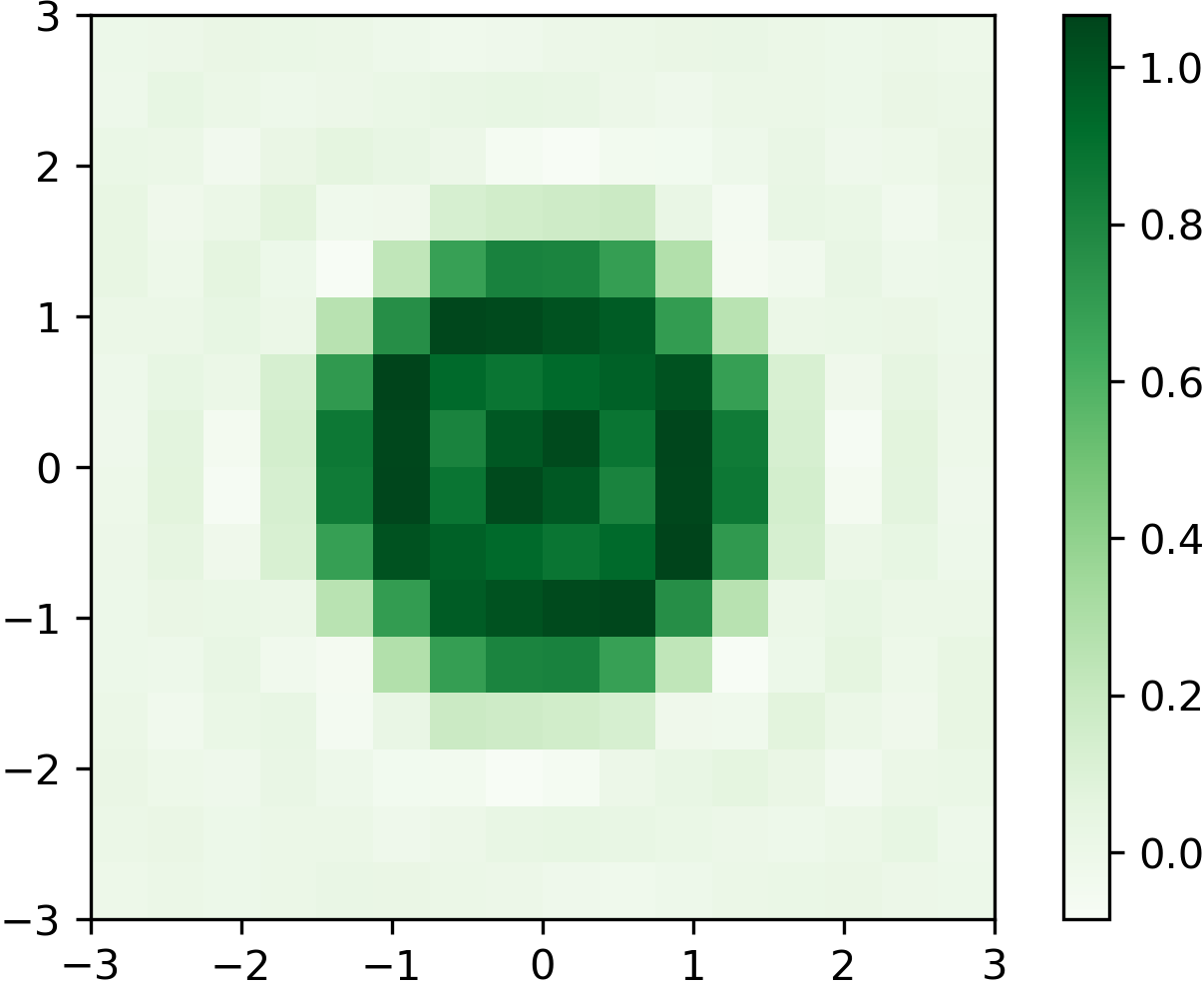}
  \subcaption{FT, $B_1$, $n=20$}
 \end{minipage}
 \begin{minipage}[b]{0.4\linewidth}
  \centering
  \includegraphics[keepaspectratio, scale=0.45]
  {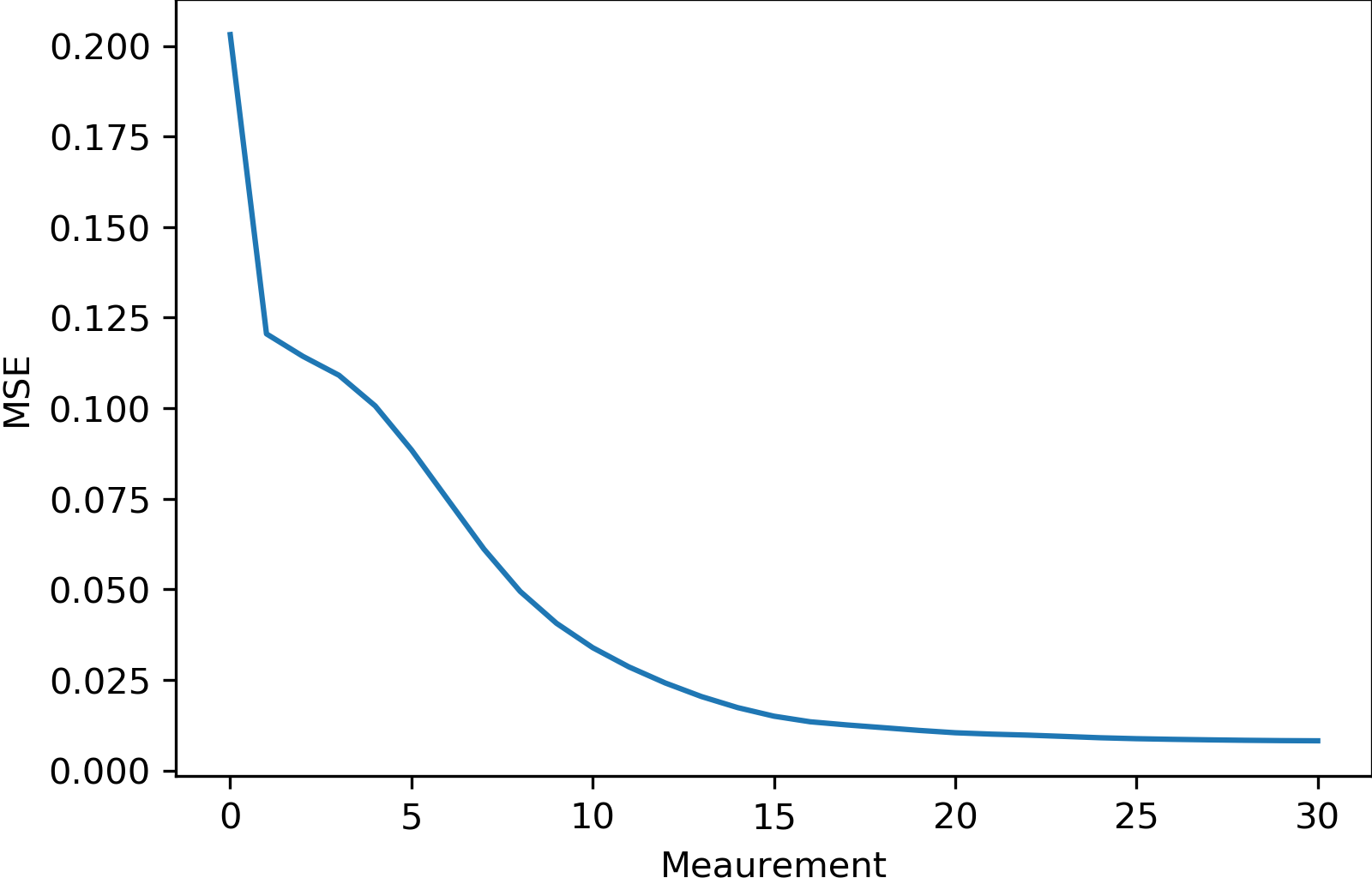}
  \subcaption{FT, $B_1$, error graph}
 \end{minipage}
\end{tabular}

\begin{tabular}{c}
\hspace{-2.5cm}
 \begin{minipage}[b]{0.4\linewidth}
  \centering
  \includegraphics[keepaspectratio, scale=0.45]
  {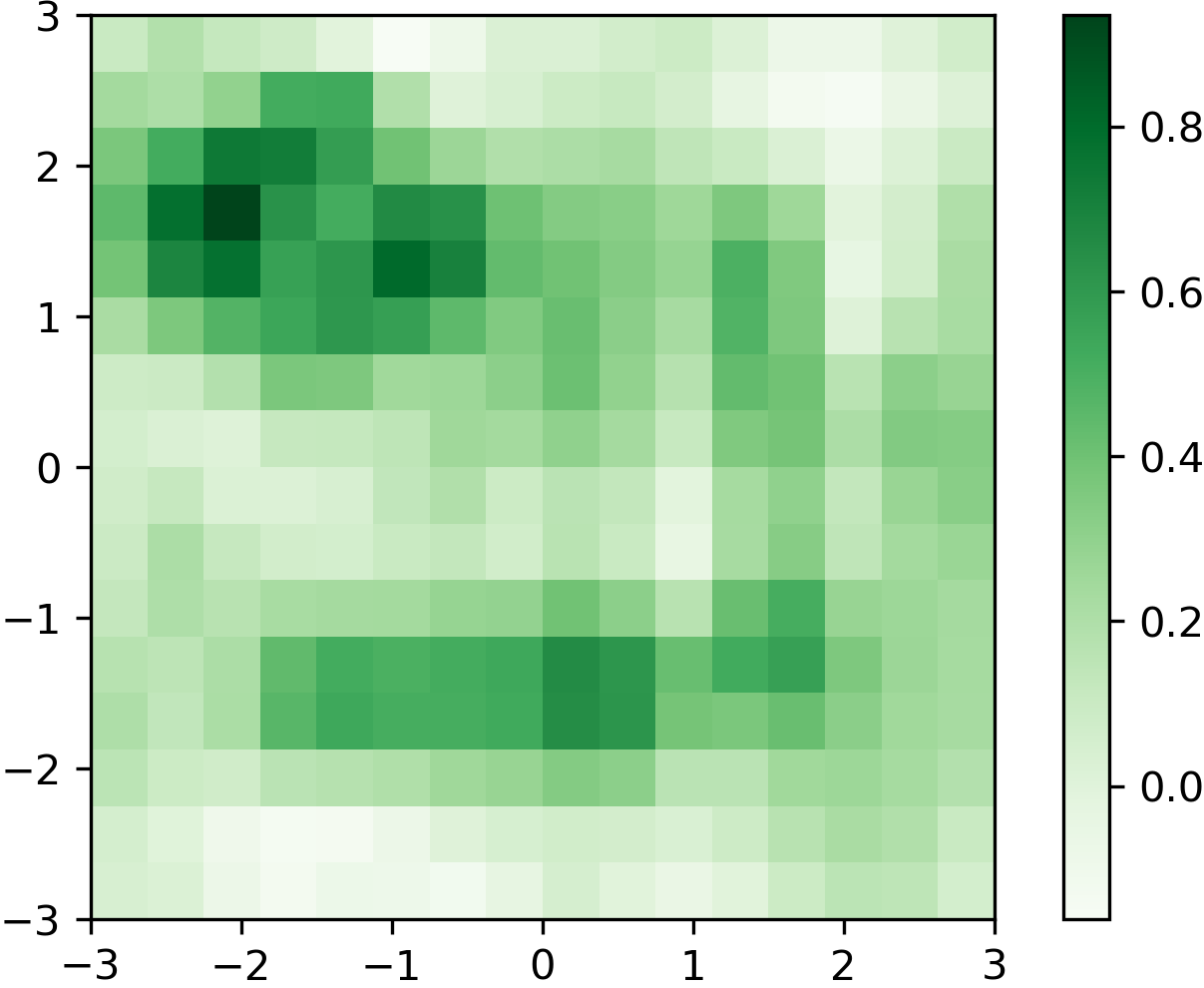}
  \subcaption{KF, $B_2$, $n=4$}
 \end{minipage}
 \begin{minipage}[b]{0.4\linewidth}
  \centering
  \includegraphics[keepaspectratio, scale=0.45]
  {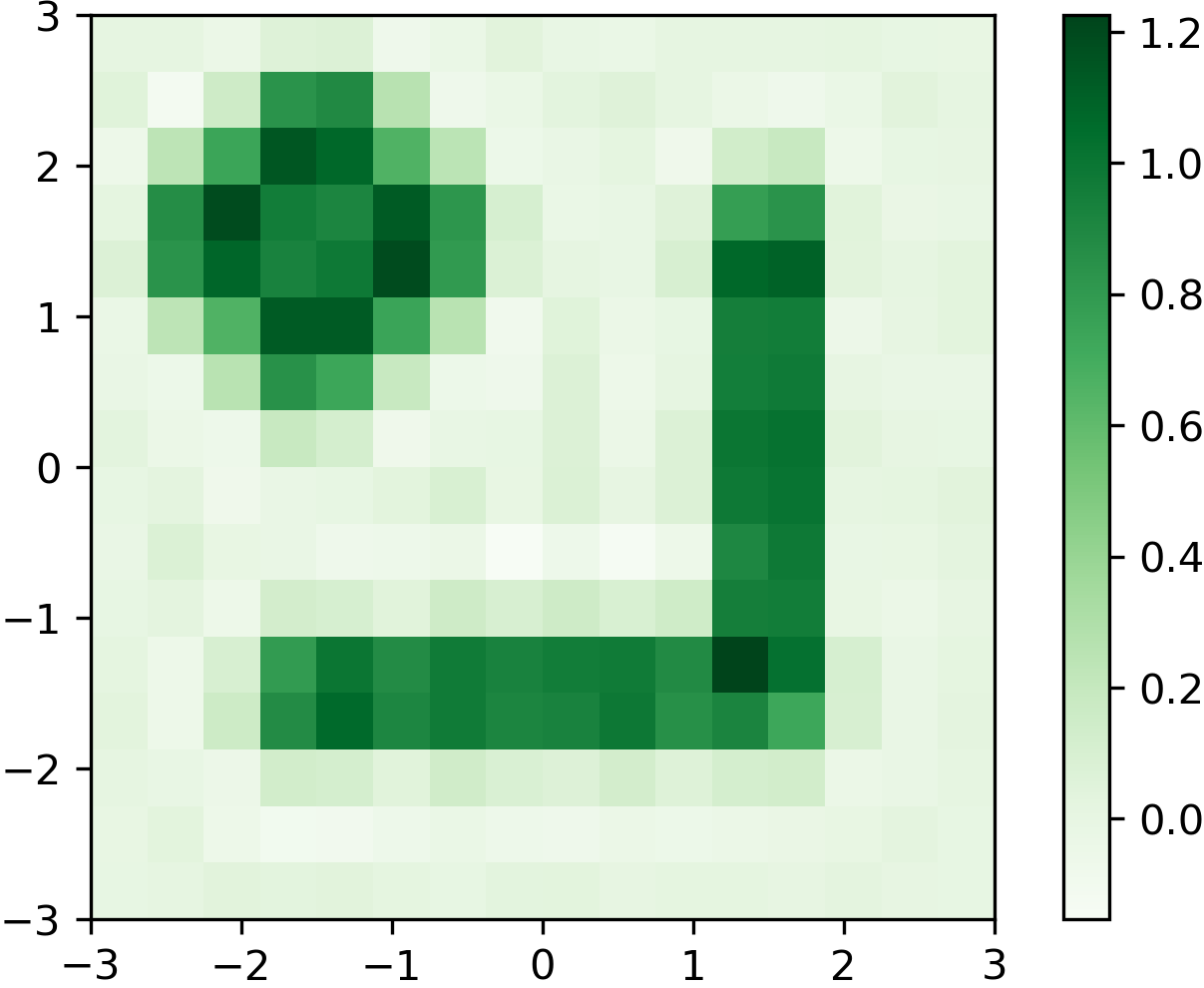}
  \subcaption{KF, $B_2$, $n=20$}
 \end{minipage}
 \begin{minipage}[b]{0.4\linewidth}
  \centering
  \includegraphics[keepaspectratio, scale=0.45]
  {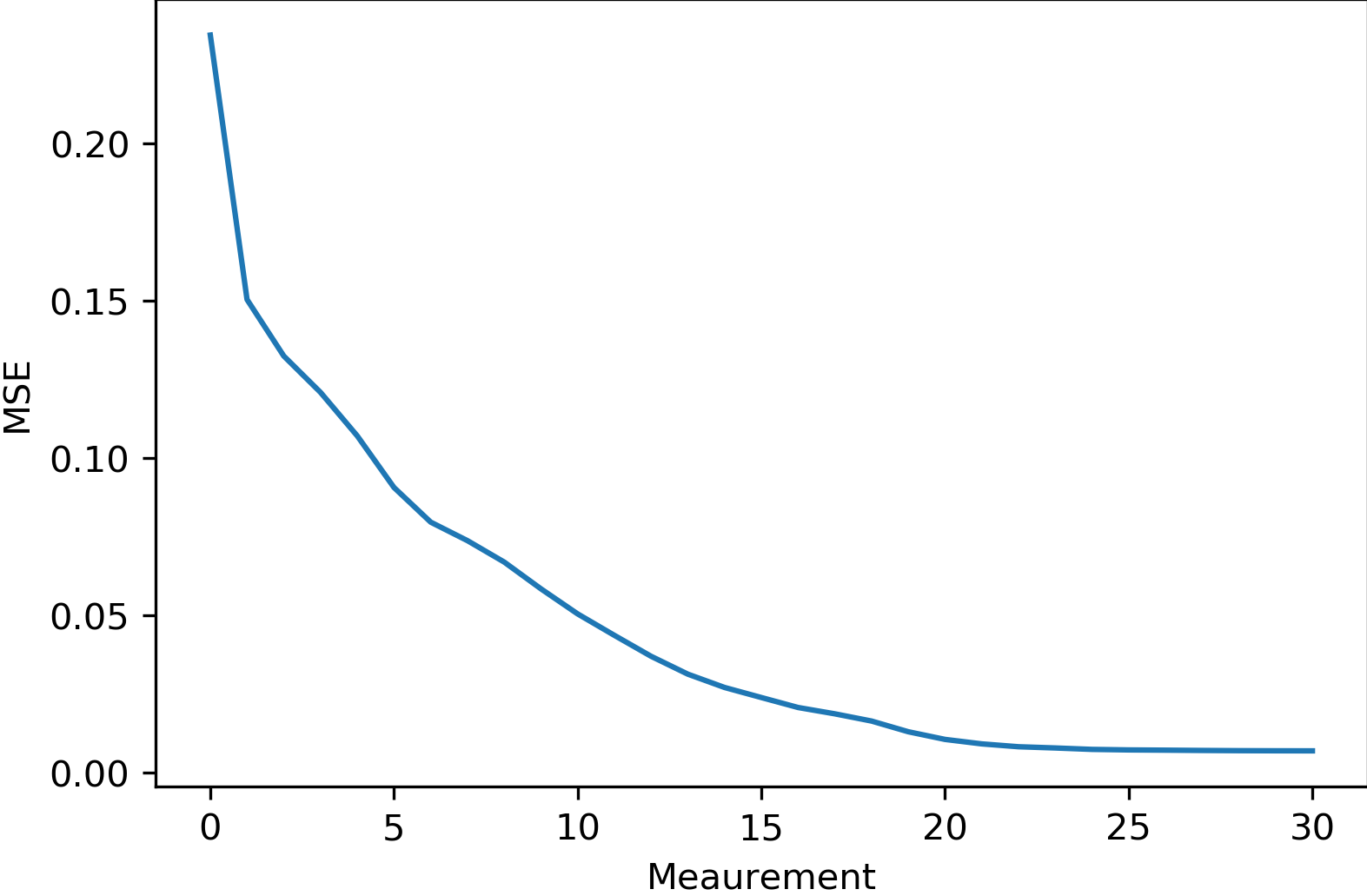}
  \subcaption{KF, $B_2$, error graph}
 \end{minipage}
\end{tabular}
\begin{tabular}{c}
\hspace{-2.5cm}
 \begin{minipage}[b]{0.4\linewidth}
  \centering
  \includegraphics[keepaspectratio, scale=0.45]
  {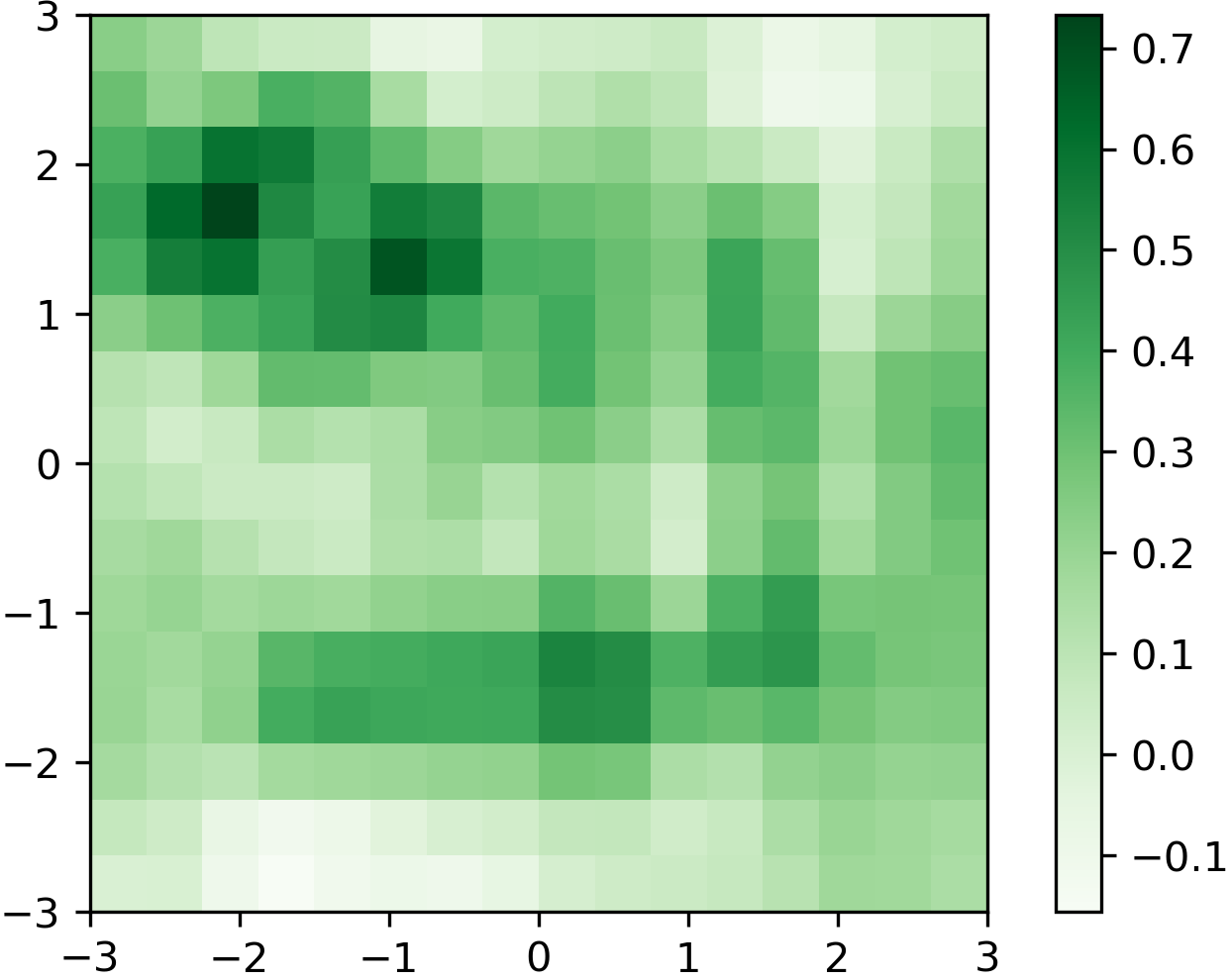}
  \subcaption{FT, $B_2$, $n=4$}
 \end{minipage}
 \begin{minipage}[b]{0.4\linewidth}
  \centering
  \includegraphics[keepaspectratio, scale=0.45]
  {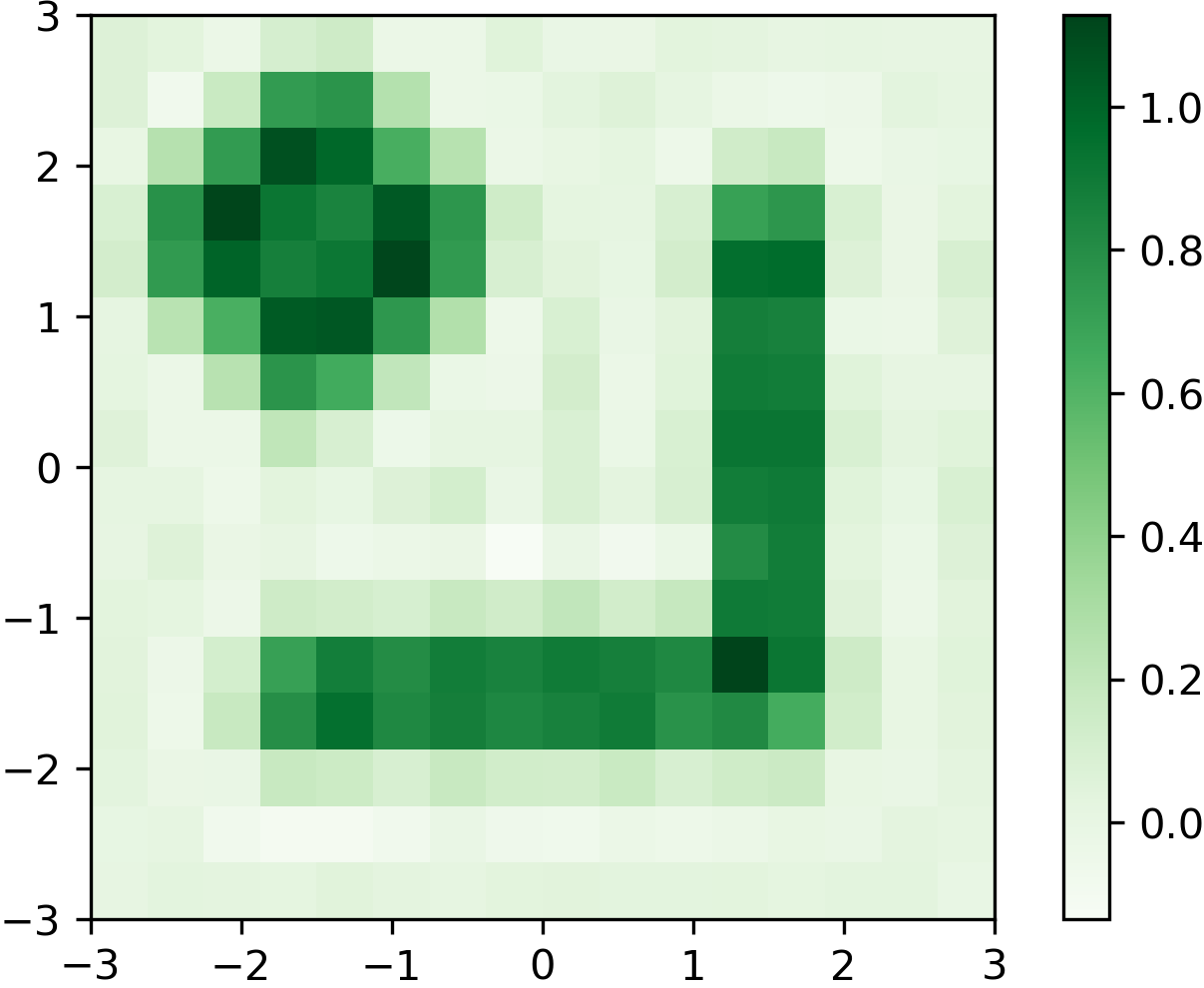}
  \subcaption{FT, $B_2$, $n=20$}
 \end{minipage}
 \begin{minipage}[b]{0.4\linewidth}
  \centering
  \includegraphics[keepaspectratio, scale=0.45]
  {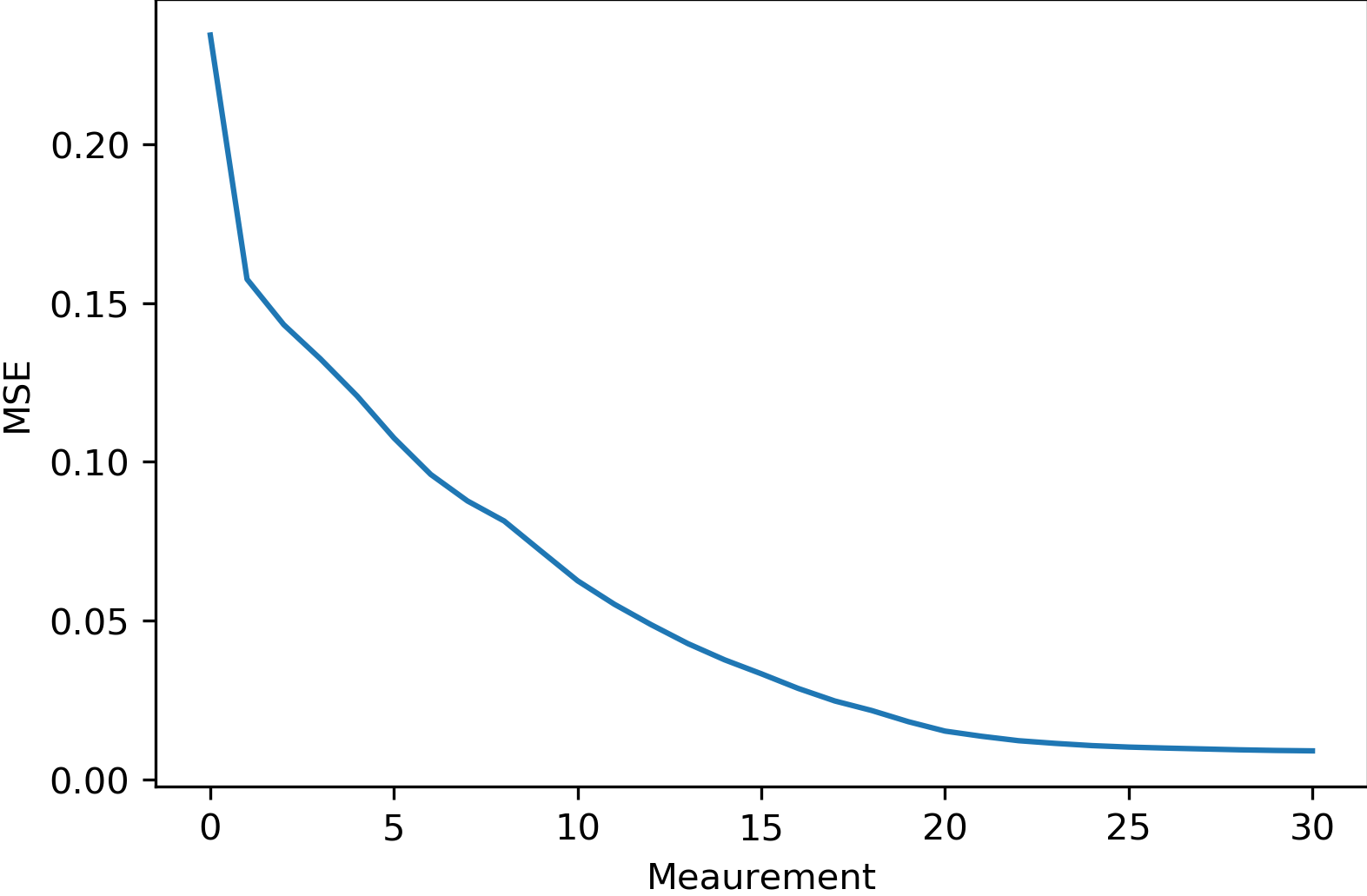}
  \subcaption{FT, $B_2$, error graph}
 \end{minipage}
\end{tabular}
\caption{the comparison of KF and FT, $k=3$, $\alpha=1$ (true data)}
\label{comparisonKFFT}
\end{figure}
\begin{figure}[h]
\begin{tabular}{c}
\hspace{-2.5cm}
 \begin{minipage}[b]{0.4\linewidth}
  \centering
  \includegraphics[keepaspectratio, scale=0.45]
  {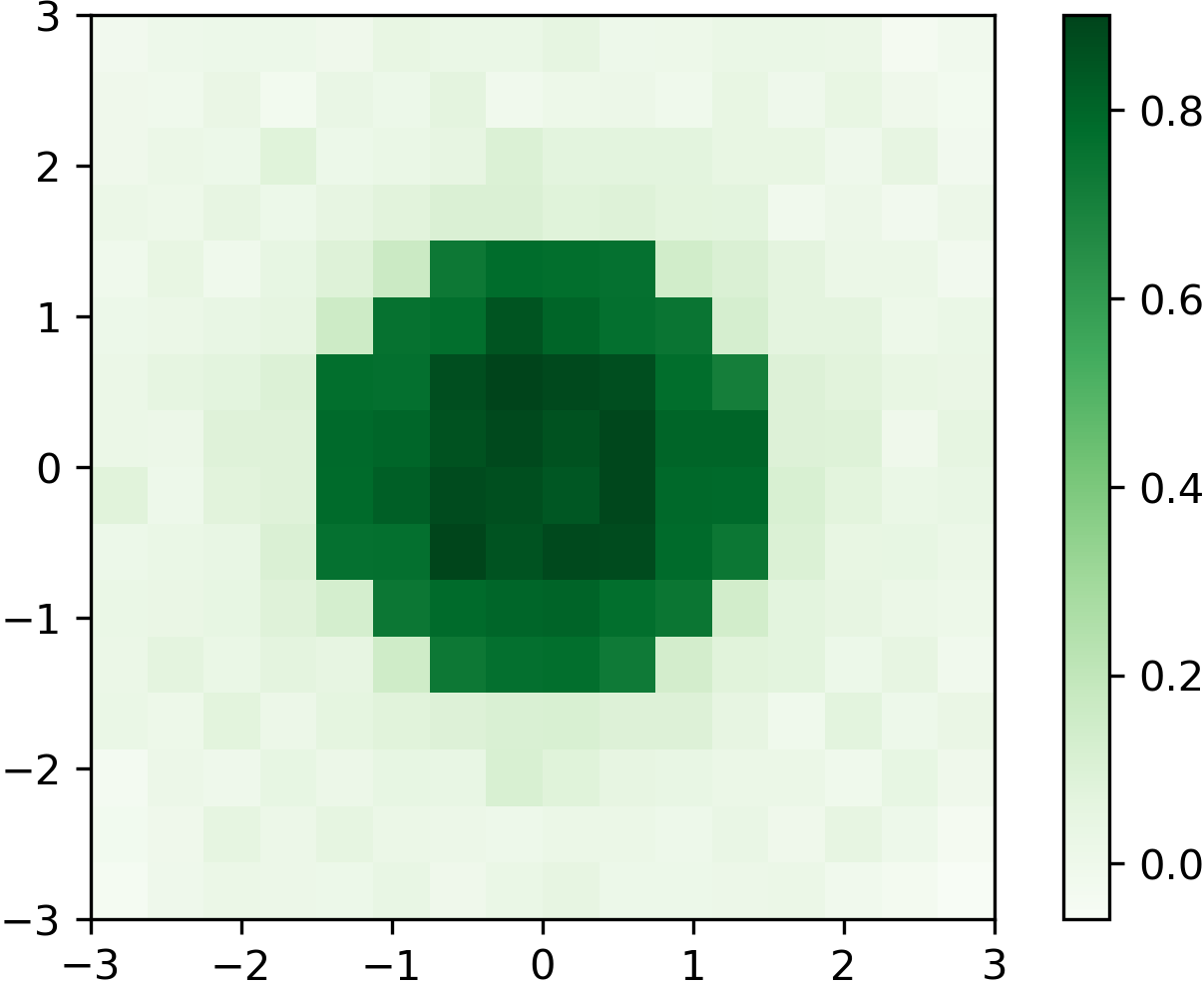}
  \subcaption{$B_1$, $k=5$, $\alpha=10$, $n=30$}
 \end{minipage}
 \begin{minipage}[b]{0.4\linewidth}
  \centering
  \includegraphics[keepaspectratio, scale=0.45]
  {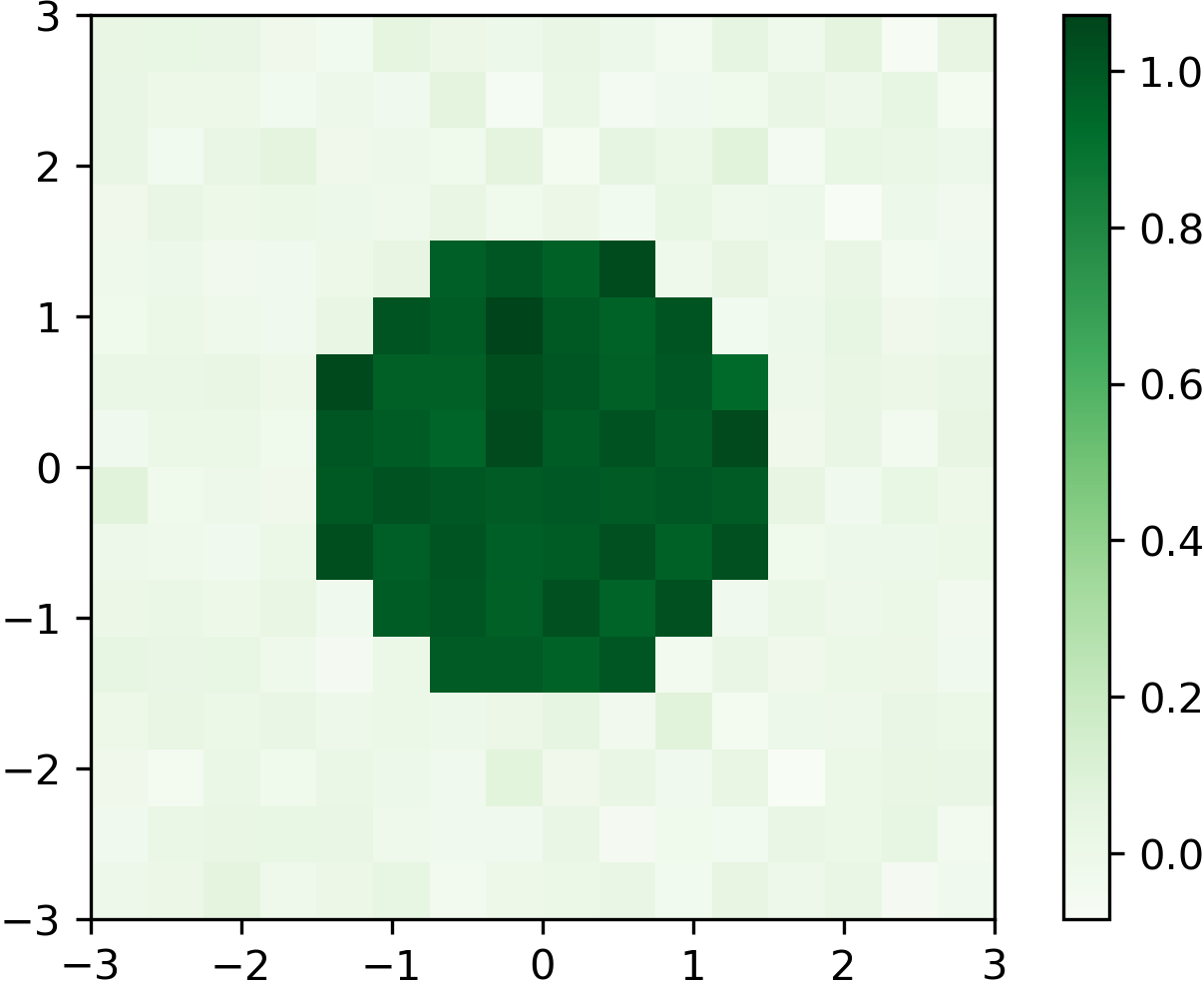}
  \subcaption{$B_1$, $k=5$, $\alpha=1\mathrm{e}-1$, $n=30$}
 \end{minipage}
 \begin{minipage}[b]{0.4\linewidth}
  \centering
  \includegraphics[keepaspectratio, scale=0.45]
  {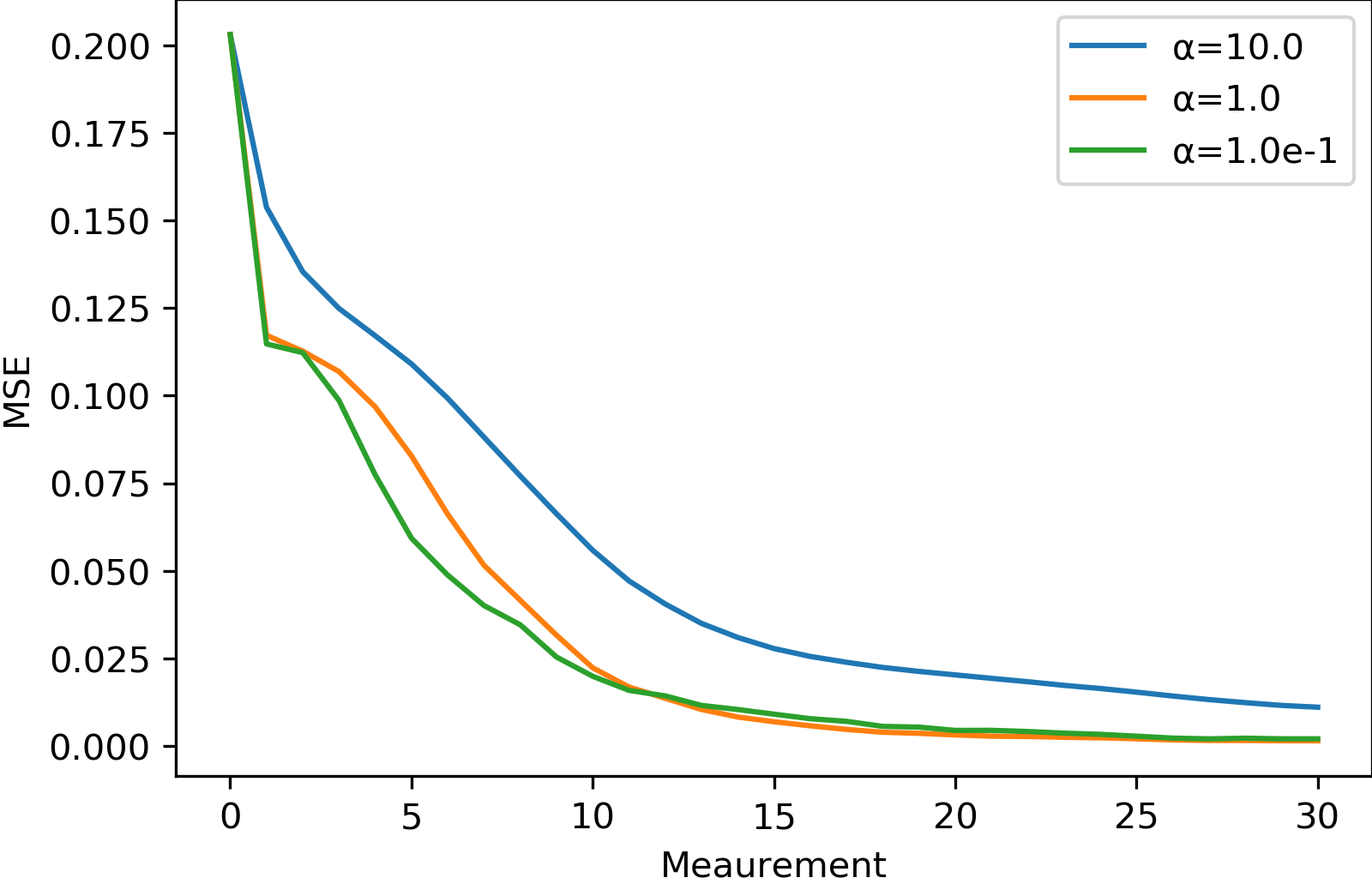}
  \subcaption{$B_1$, $k=5$, error graph}
 \end{minipage}
\end{tabular}
\begin{tabular}{c}
\hspace{-2.5cm}
 \begin{minipage}[b]{0.4\linewidth}
  \centering
  \includegraphics[keepaspectratio, scale=0.45]
  {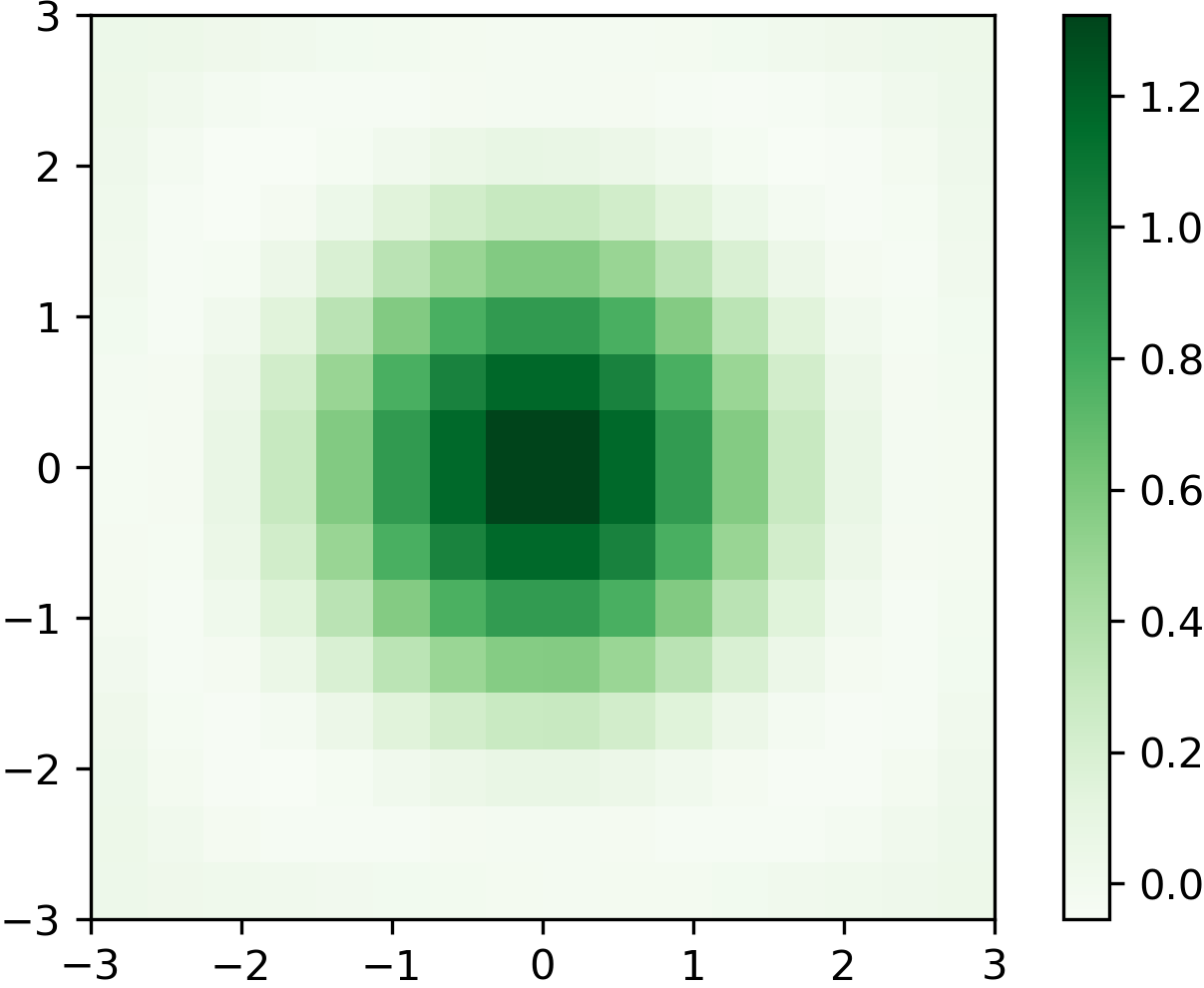}
  \subcaption{$B_1$, $k=1$, $\alpha=10$, $n=30$}
 \end{minipage}
 \begin{minipage}[b]{0.4\linewidth}
  \centering
  \includegraphics[keepaspectratio, scale=0.45]
  {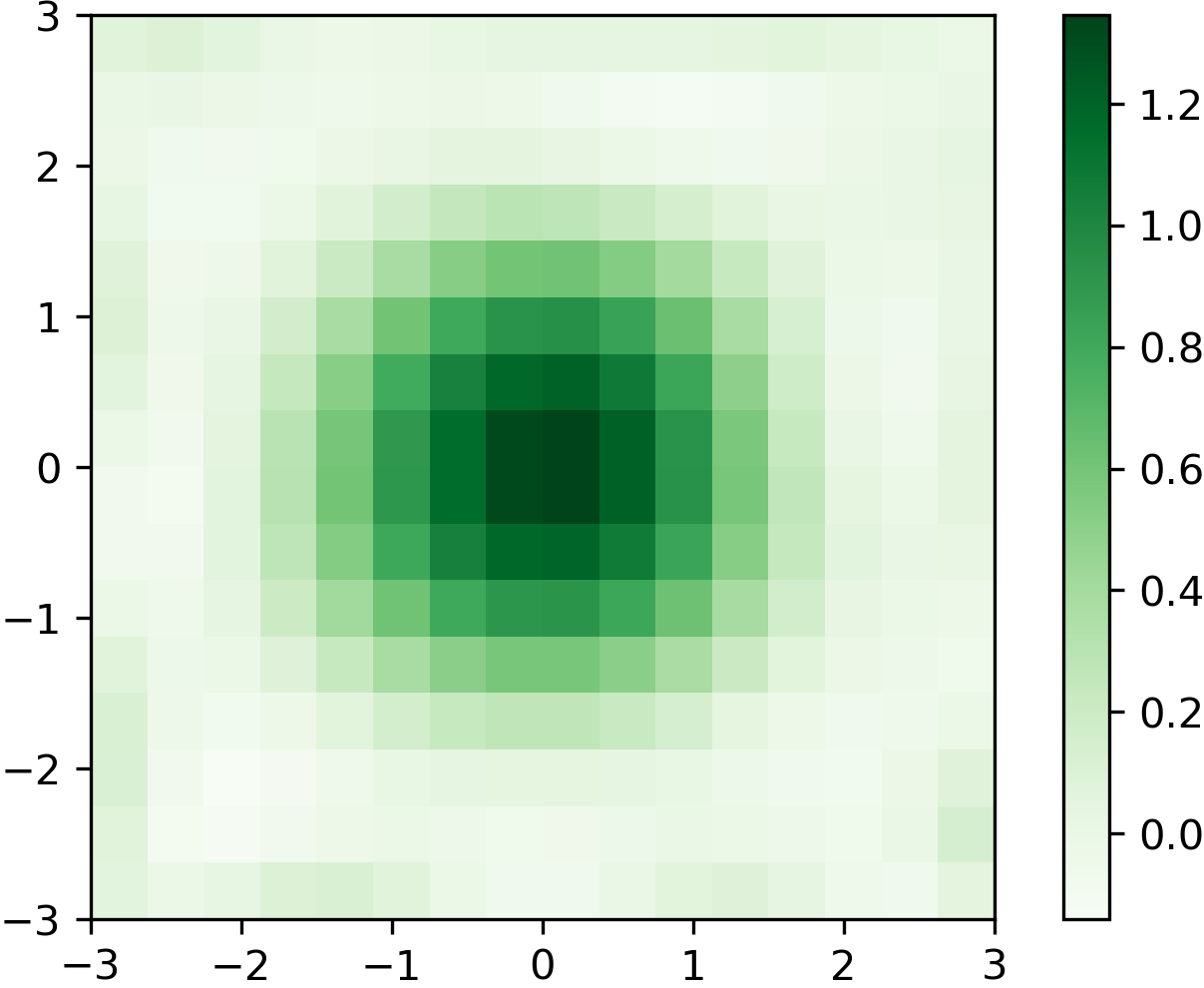}
  \subcaption{$B_1$, $k=1$, $\alpha=1\mathrm{e}-1$, $n=30$}
 \end{minipage}
 \begin{minipage}[b]{0.4\linewidth}
  \centering
  \includegraphics[keepaspectratio, scale=0.45]
  {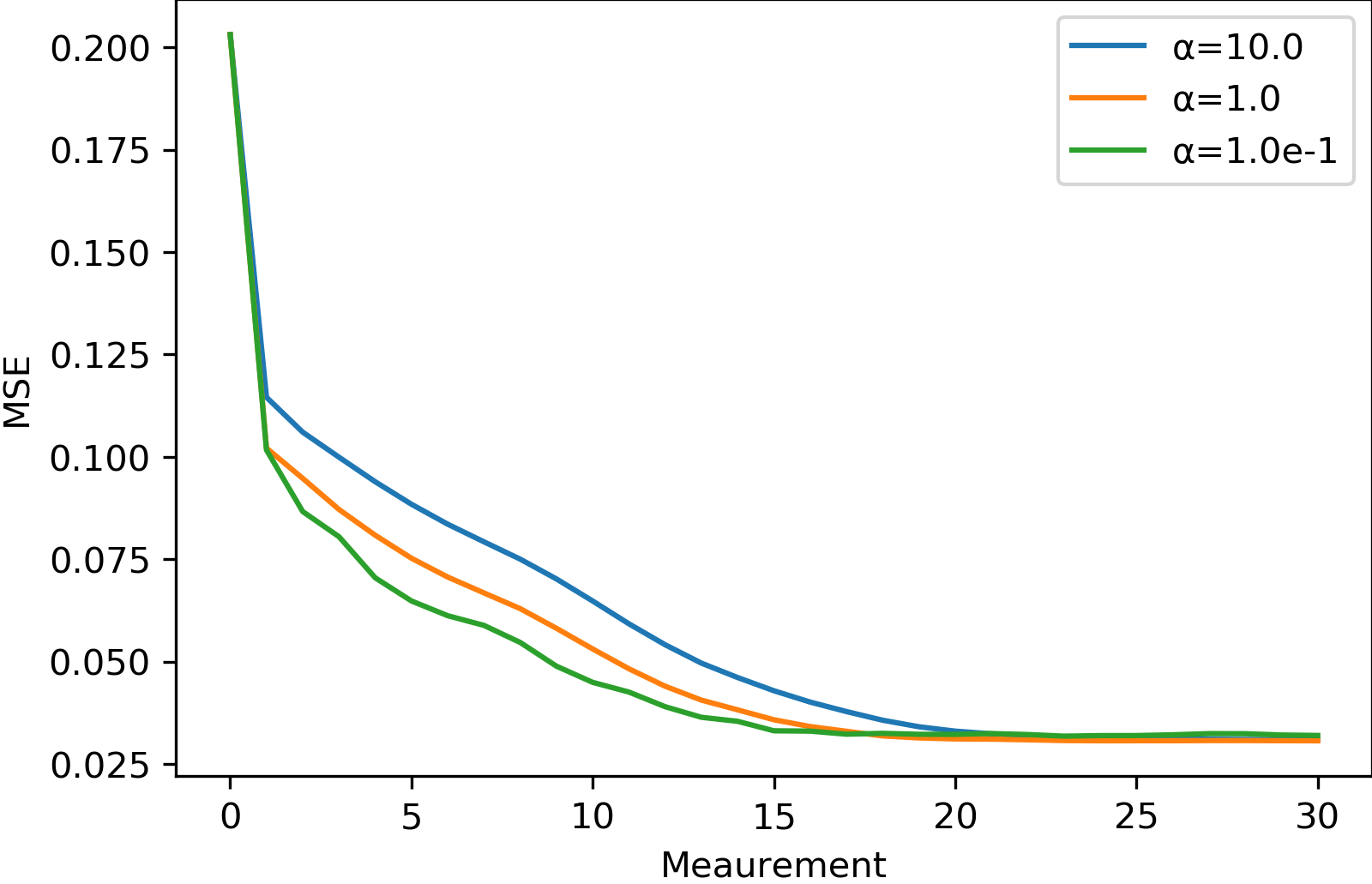}
  \subcaption{$B_1$, $k=1$, error graph}
 \end{minipage}
\end{tabular}
\begin{tabular}{c}
\hspace{-2.5cm}
 \begin{minipage}[b]{0.4\linewidth}
  \centering
  \includegraphics[keepaspectratio, scale=0.45]
  {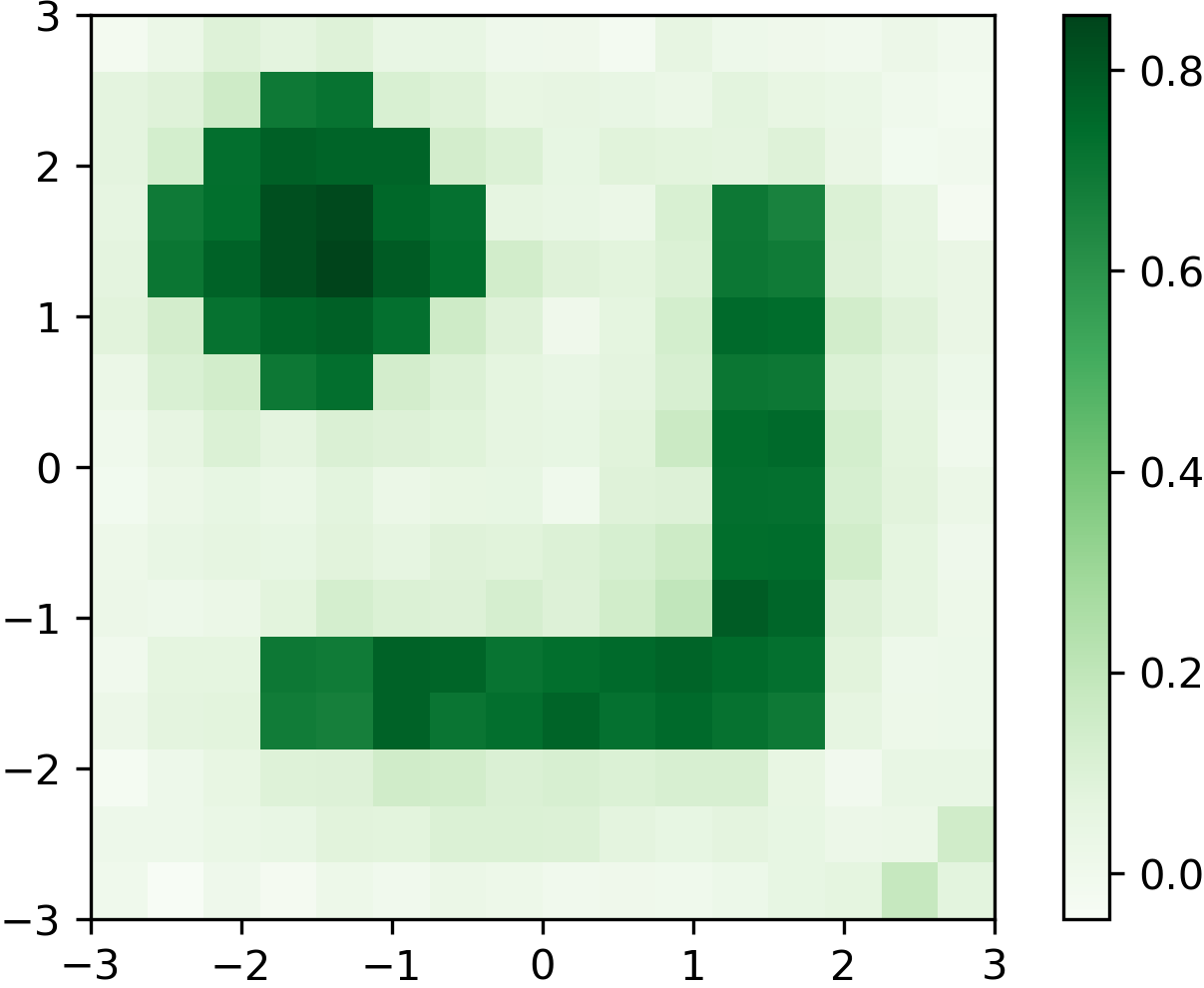}
  \subcaption{$B_2$, $k=5$, $\alpha=10$, $n=30$}
 \end{minipage}
 \begin{minipage}[b]{0.4\linewidth}
  \centering
  \includegraphics[keepaspectratio, scale=0.45]
  {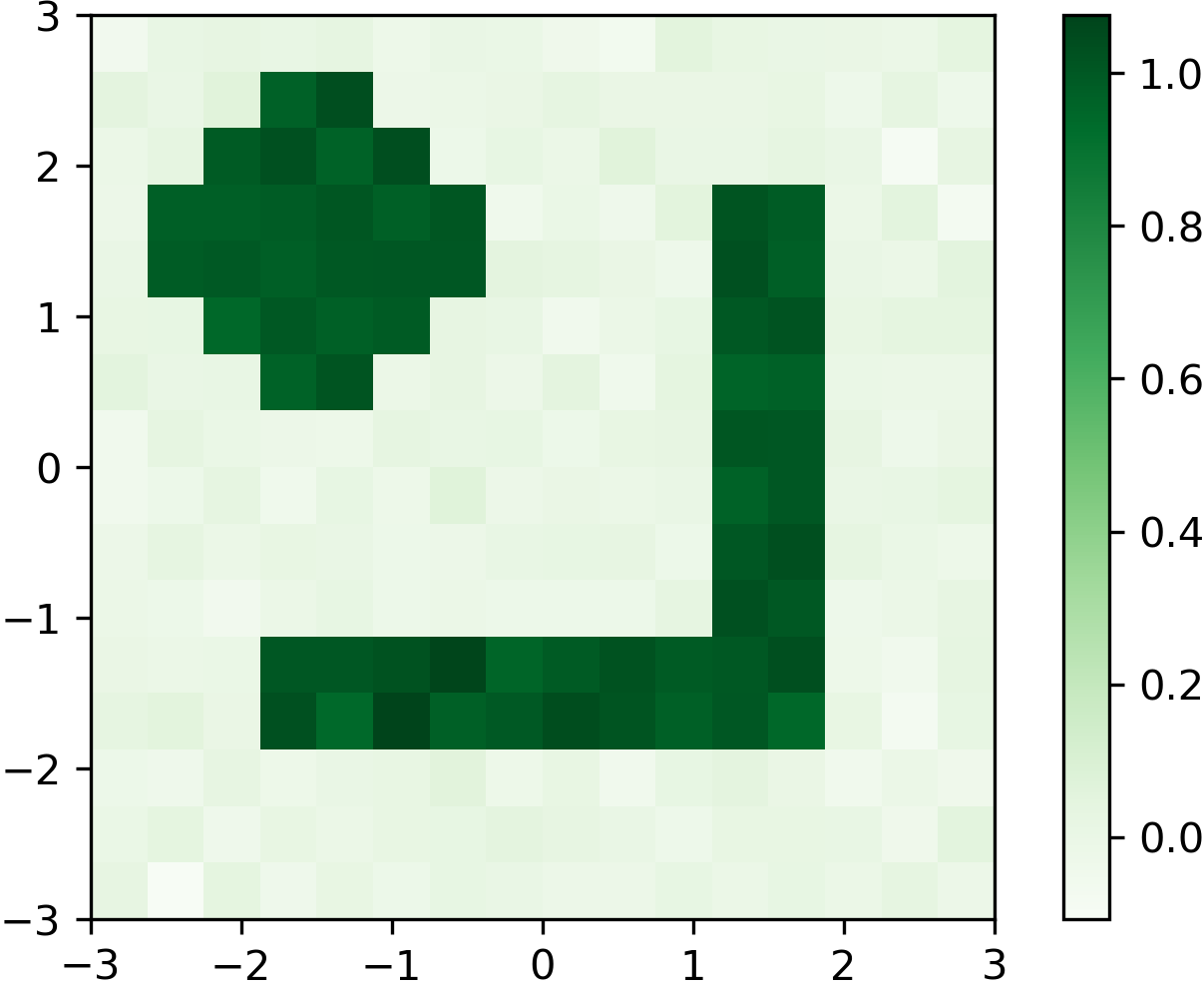}
  \subcaption{$B_2$, $k=5$, $\alpha=1\mathrm{e}-1$, $n=30$}
 \end{minipage}
 \begin{minipage}[b]{0.4\linewidth}
  \centering
  \includegraphics[keepaspectratio, scale=0.45]
  {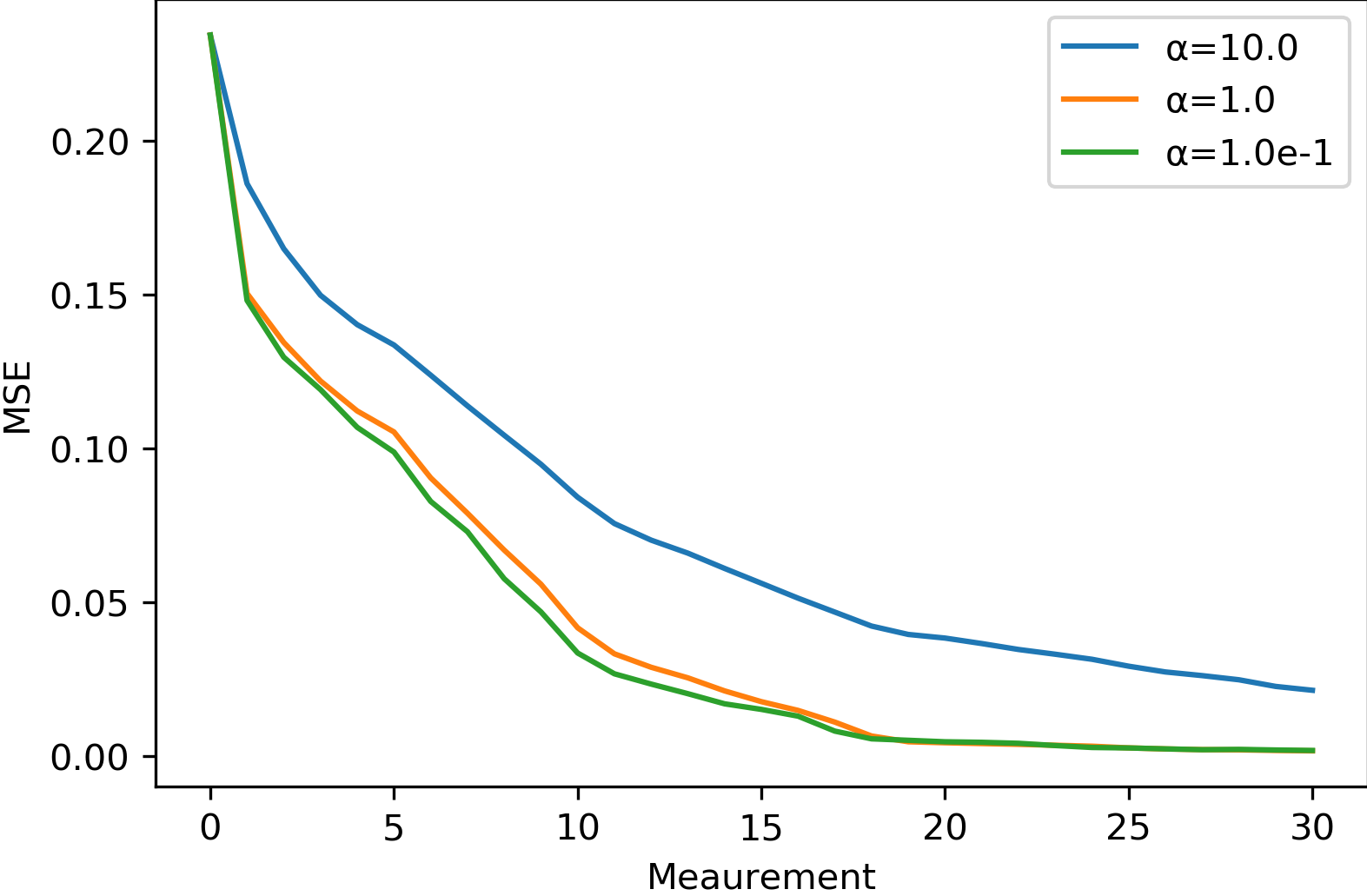}
  \subcaption{$B_2$, $k=5$, error graph}
 \end{minipage}
\end{tabular}
\begin{tabular}{c}
\hspace{-2.5cm}
 \begin{minipage}[b]{0.4\linewidth}
  \centering
  \includegraphics[keepaspectratio, scale=0.45]
  {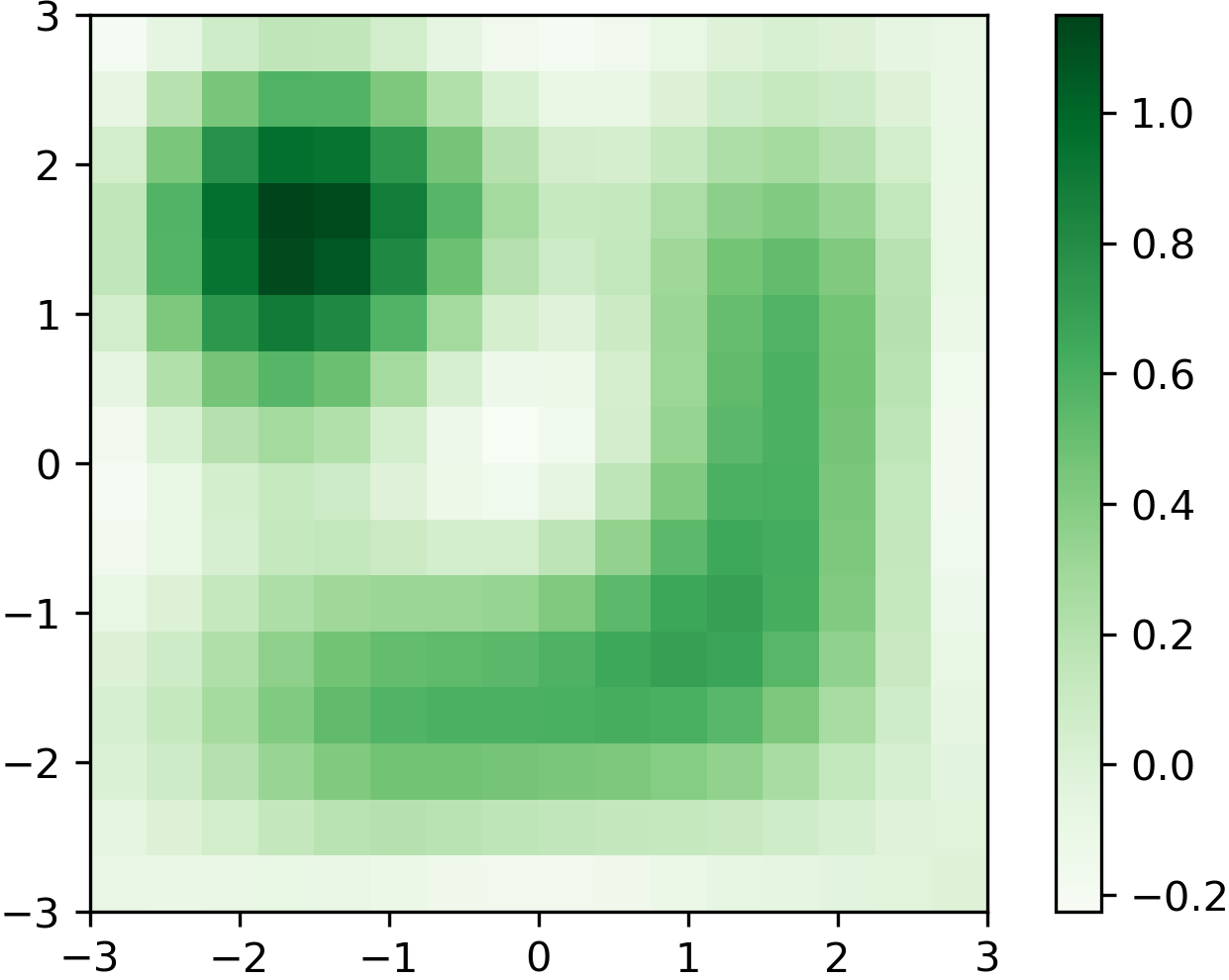}
  \subcaption{$B_2$, $k=1$, $\alpha=10$, $n=30$}
 \end{minipage}
 \begin{minipage}[b]{0.4\linewidth}
  \centering
  \includegraphics[keepaspectratio, scale=0.45]
  {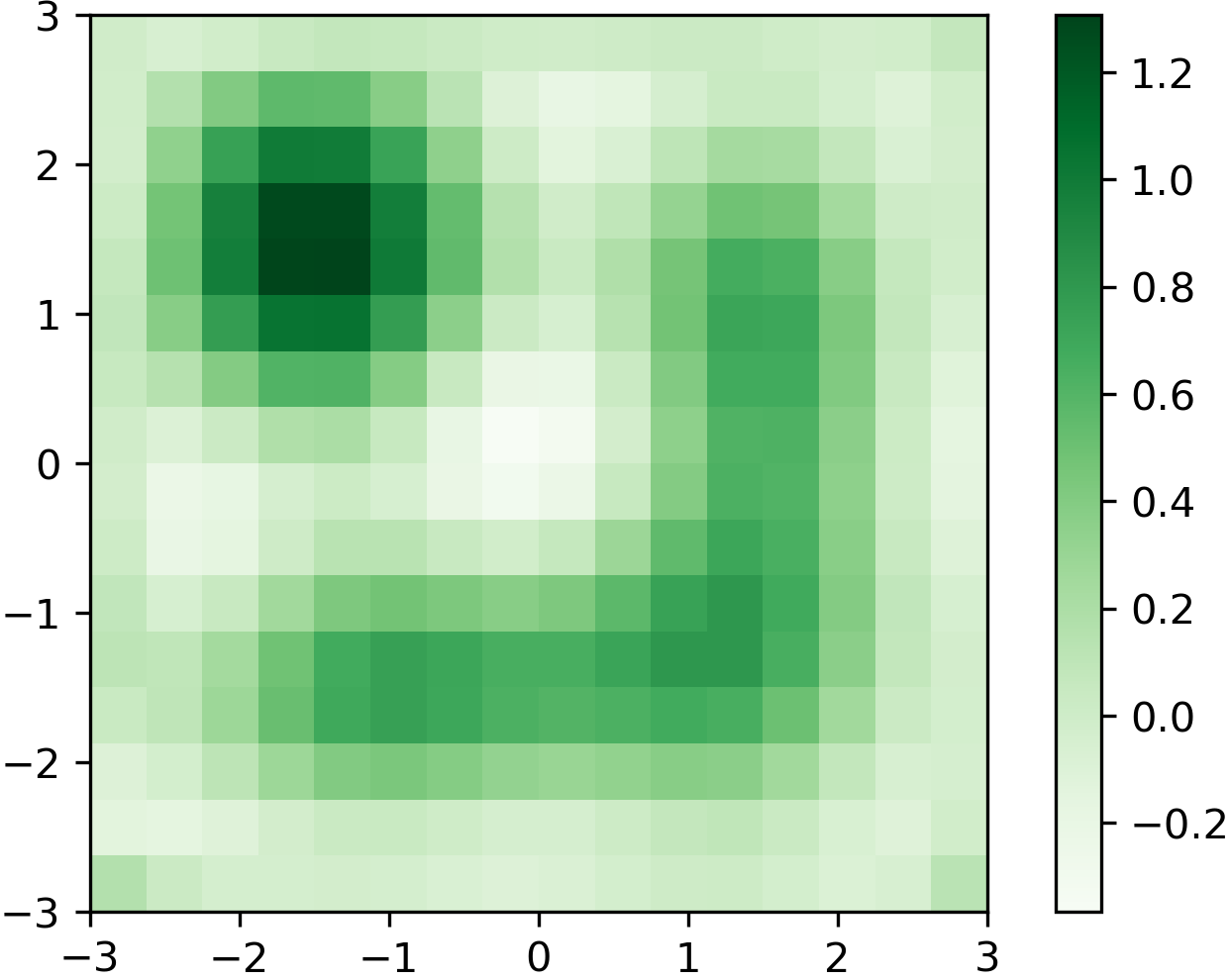}
  \subcaption{$B_2$, $k=1$, $\alpha=1\mathrm{e}-1$, $n=30$}
 \end{minipage}
 \begin{minipage}[b]{0.4\linewidth}
  \centering
  \includegraphics[keepaspectratio, scale=0.45]
  {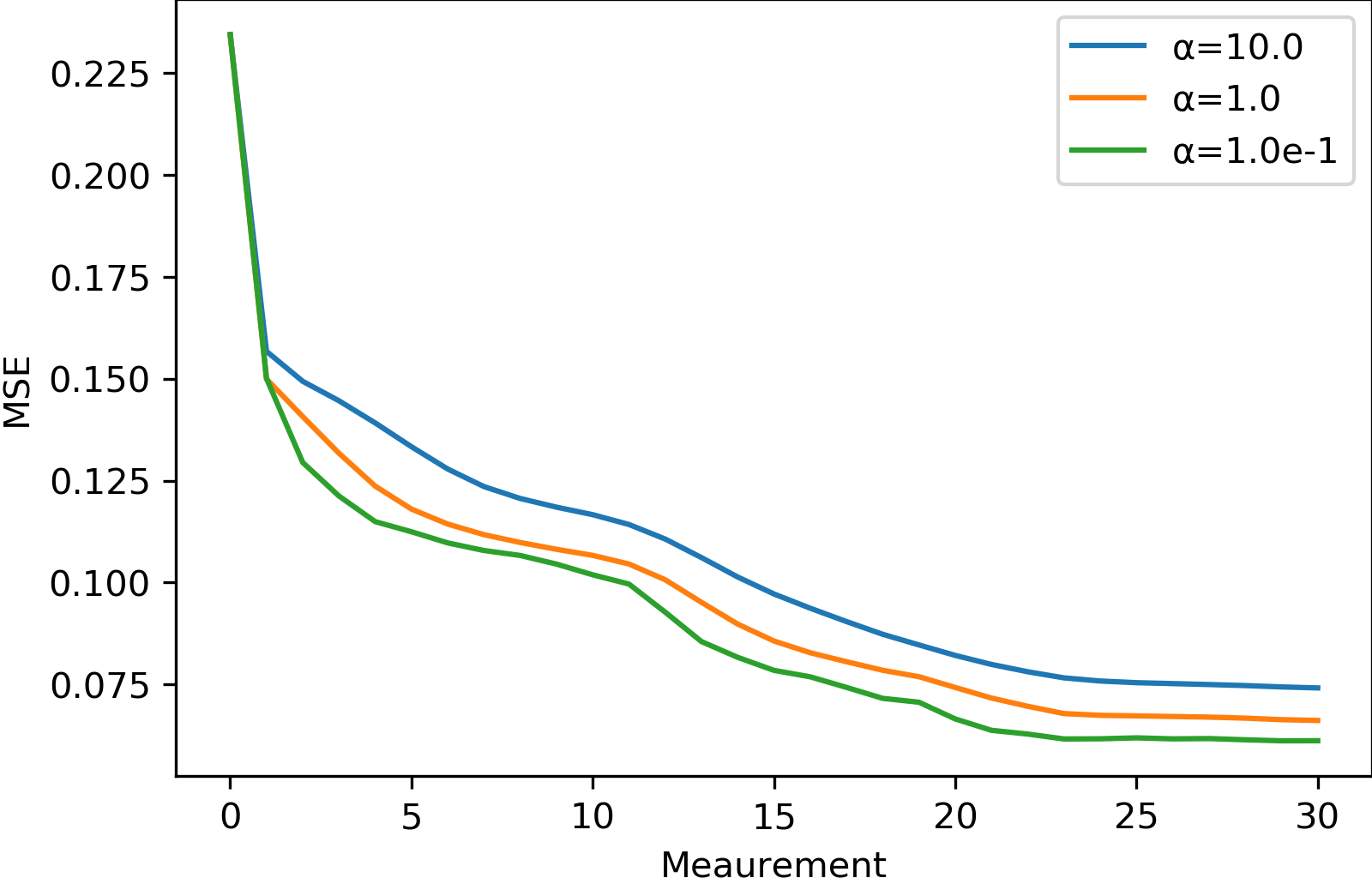}
  \subcaption{$B_2$, $k=1$, error graph}
 \end{minipage}
\end{tabular}
\caption{KF reconstruction for different $k$ and $\alpha$ (nosiy $\sigma= 0.1$)}
\label{KFreconstruction}
\end{figure}
\begin{figure}[h]
\begin{tabular}{c}
\hspace{-2.5cm}
 \begin{minipage}[b]{0.4\linewidth}
  \centering
  \includegraphics[keepaspectratio, scale=0.45]
  {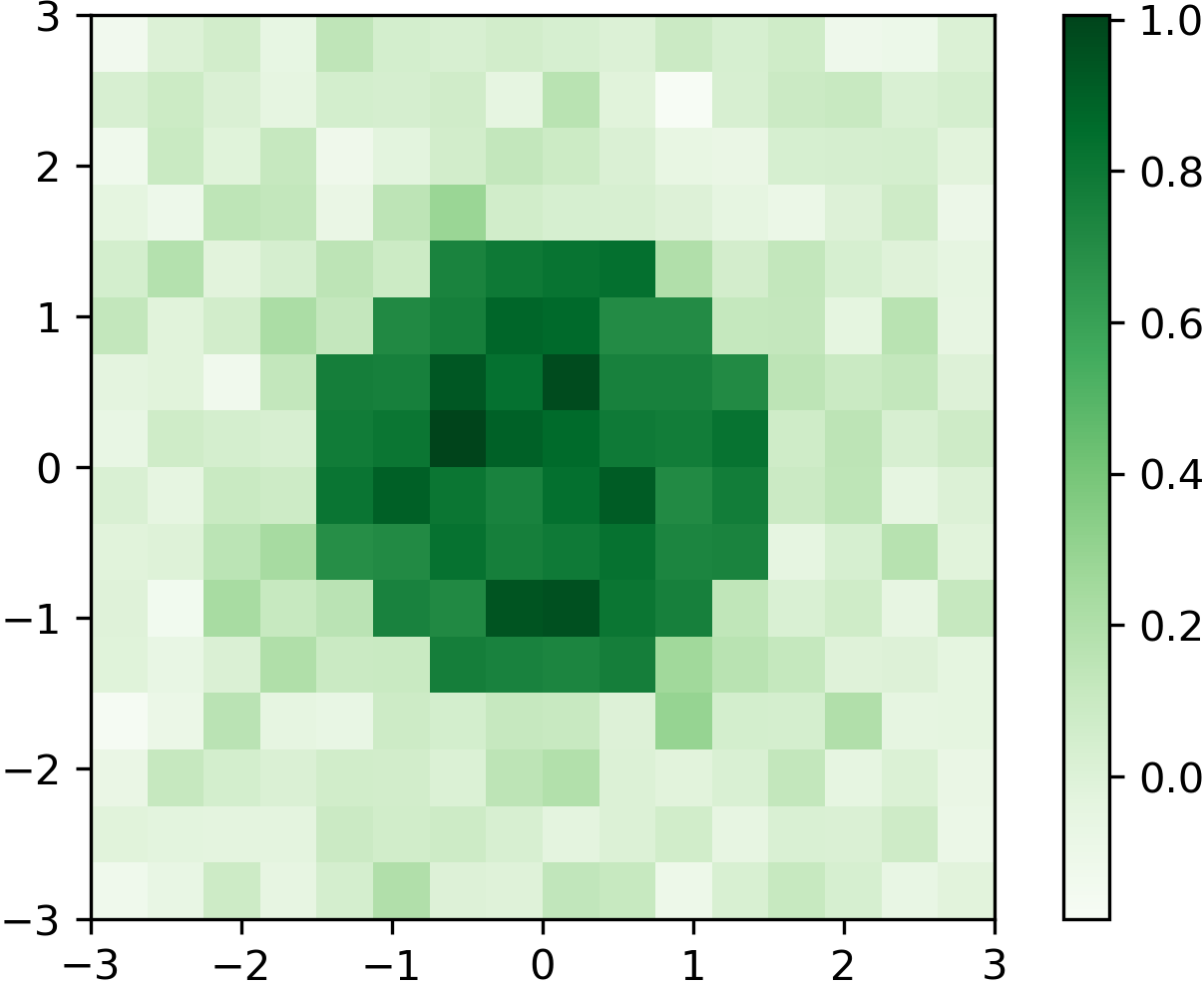}
  \subcaption{$B_1$, $k=5$, $\alpha=10$, $n=30$}
 \end{minipage}
 \begin{minipage}[b]{0.4\linewidth}
  \centering
  \includegraphics[keepaspectratio, scale=0.45]
  {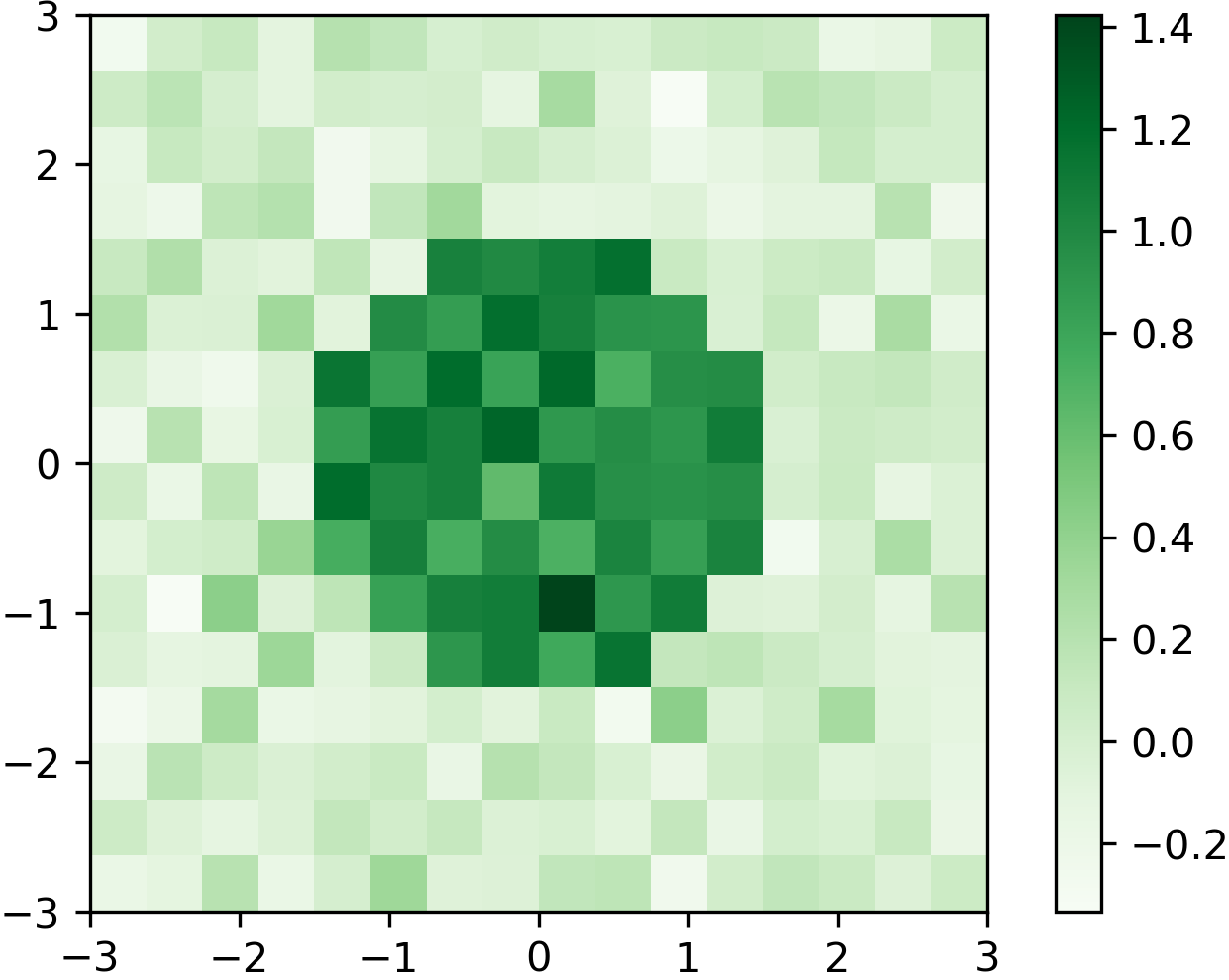}
  \subcaption{$B_1$, $k=5$, $\alpha=1\mathrm{e}-1$, $n=30$}
 \end{minipage}
 \begin{minipage}[b]{0.4\linewidth}
  \centering
  \includegraphics[keepaspectratio, scale=0.45]
  {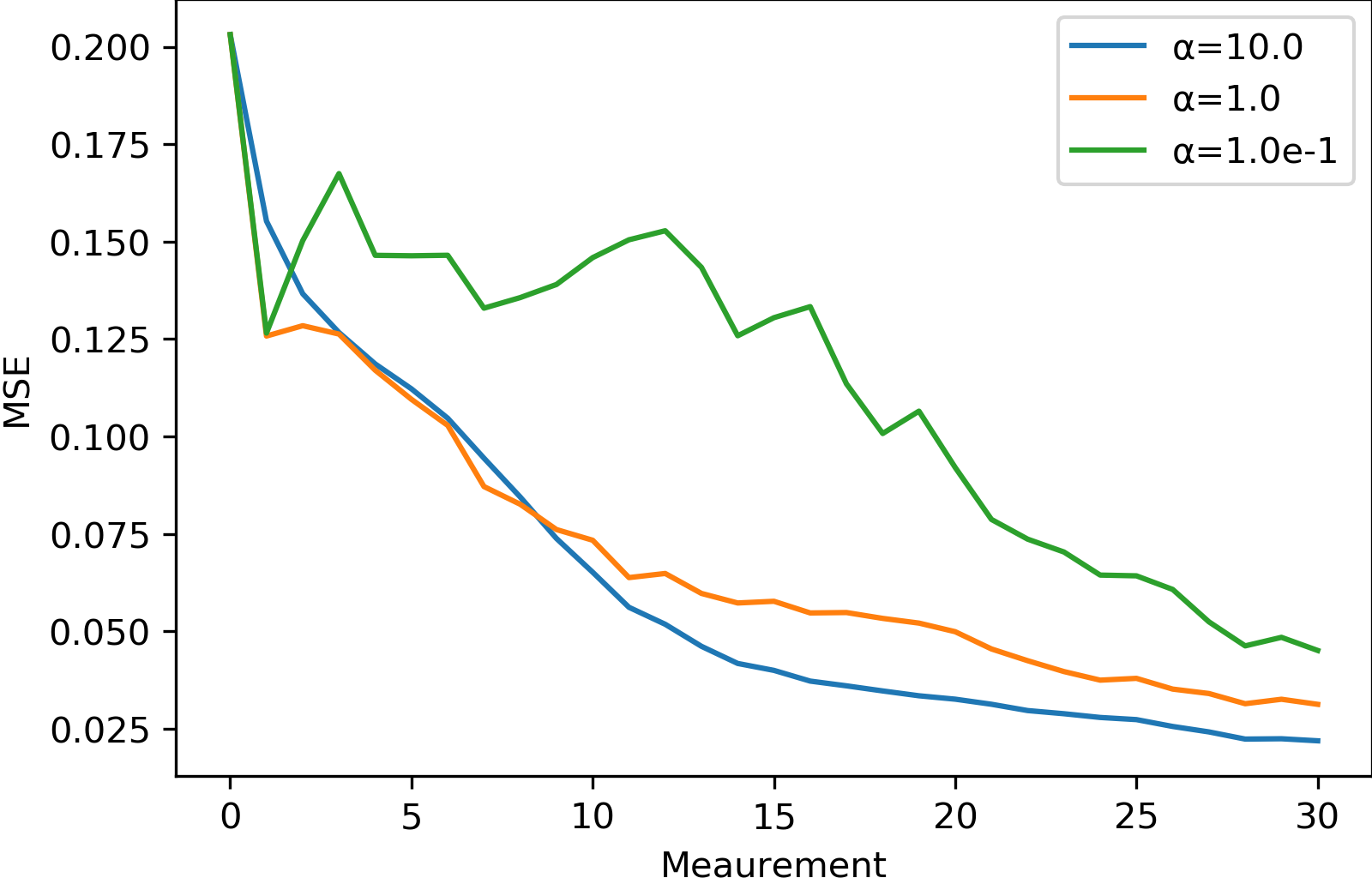}
  \subcaption{$B_1$, $k=5$, error graph}
 \end{minipage}
\end{tabular}
\begin{tabular}{c}
\hspace{-2.5cm}
 \begin{minipage}[b]{0.4\linewidth}
  \centering
  \includegraphics[keepaspectratio, scale=0.45]
  {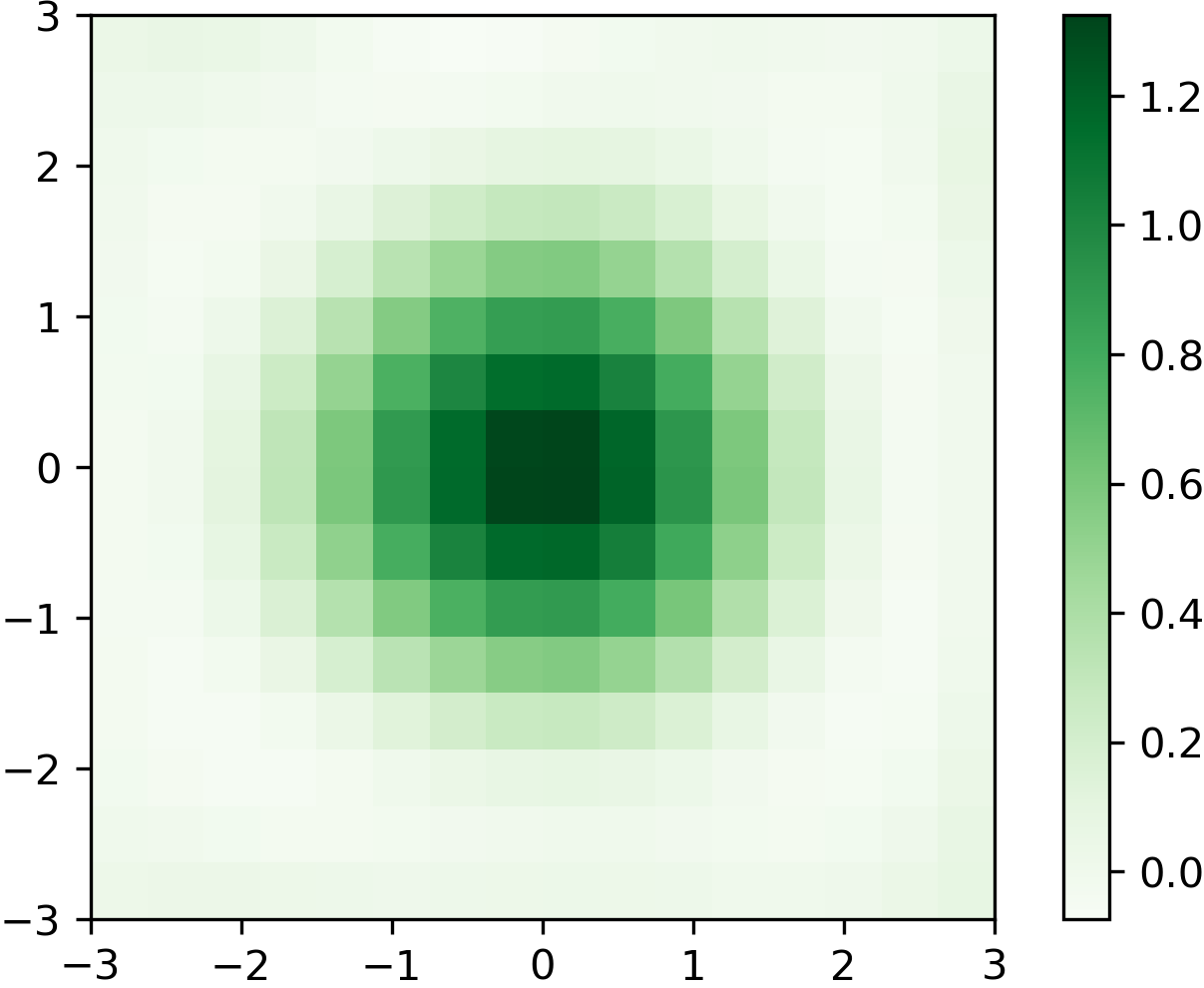}
  \subcaption{$B_1$, $k=1$, $\alpha=10$, $n=30$}
 \end{minipage}
 \begin{minipage}[b]{0.4\linewidth}
  \centering
  \includegraphics[keepaspectratio, scale=0.45]
  {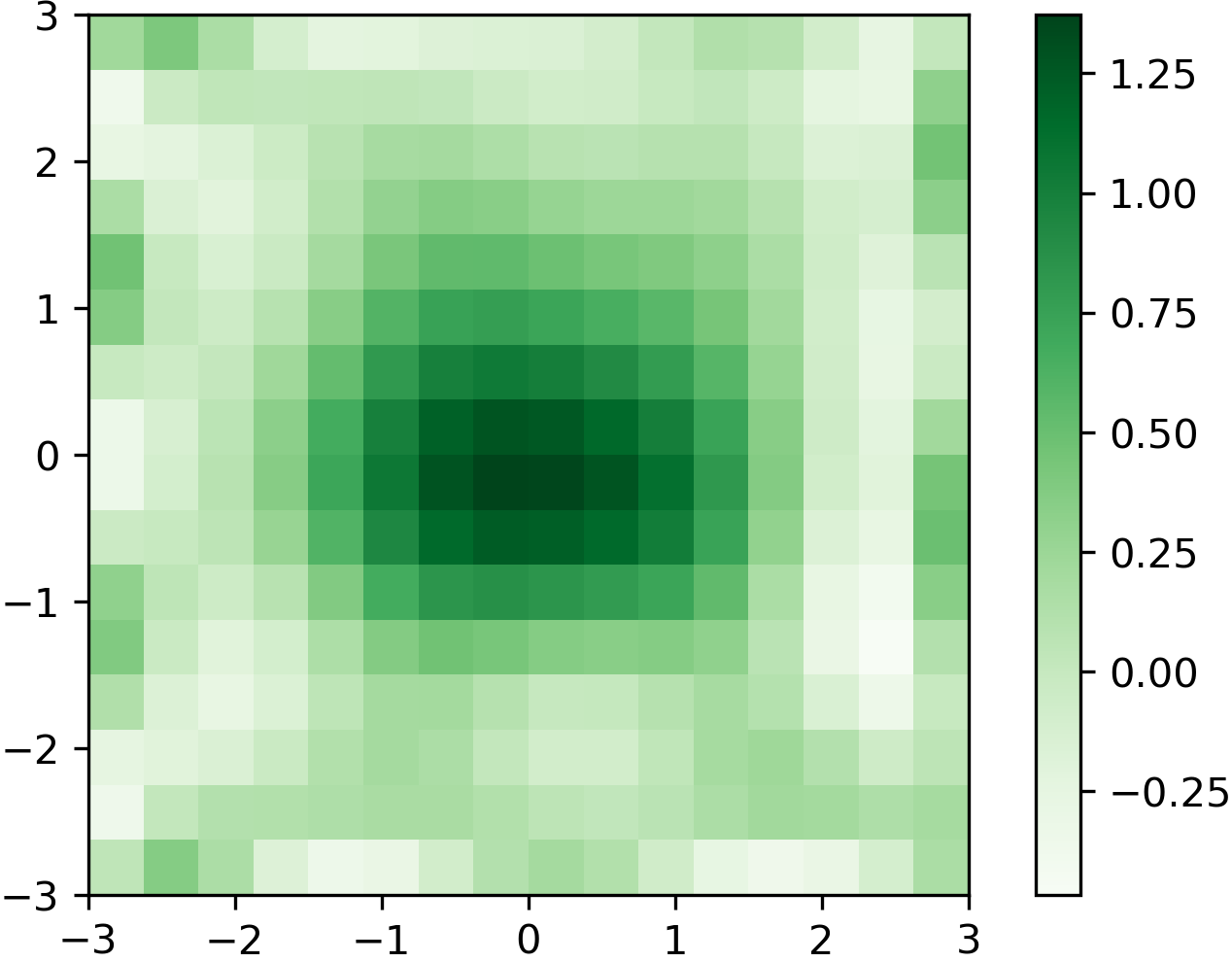}
  \subcaption{$B_1$, $k=1$, $\alpha=1\mathrm{e}-1$, $n=30$}
 \end{minipage}
 \begin{minipage}[b]{0.4\linewidth}
  \centering
  \includegraphics[keepaspectratio, scale=0.45]
  {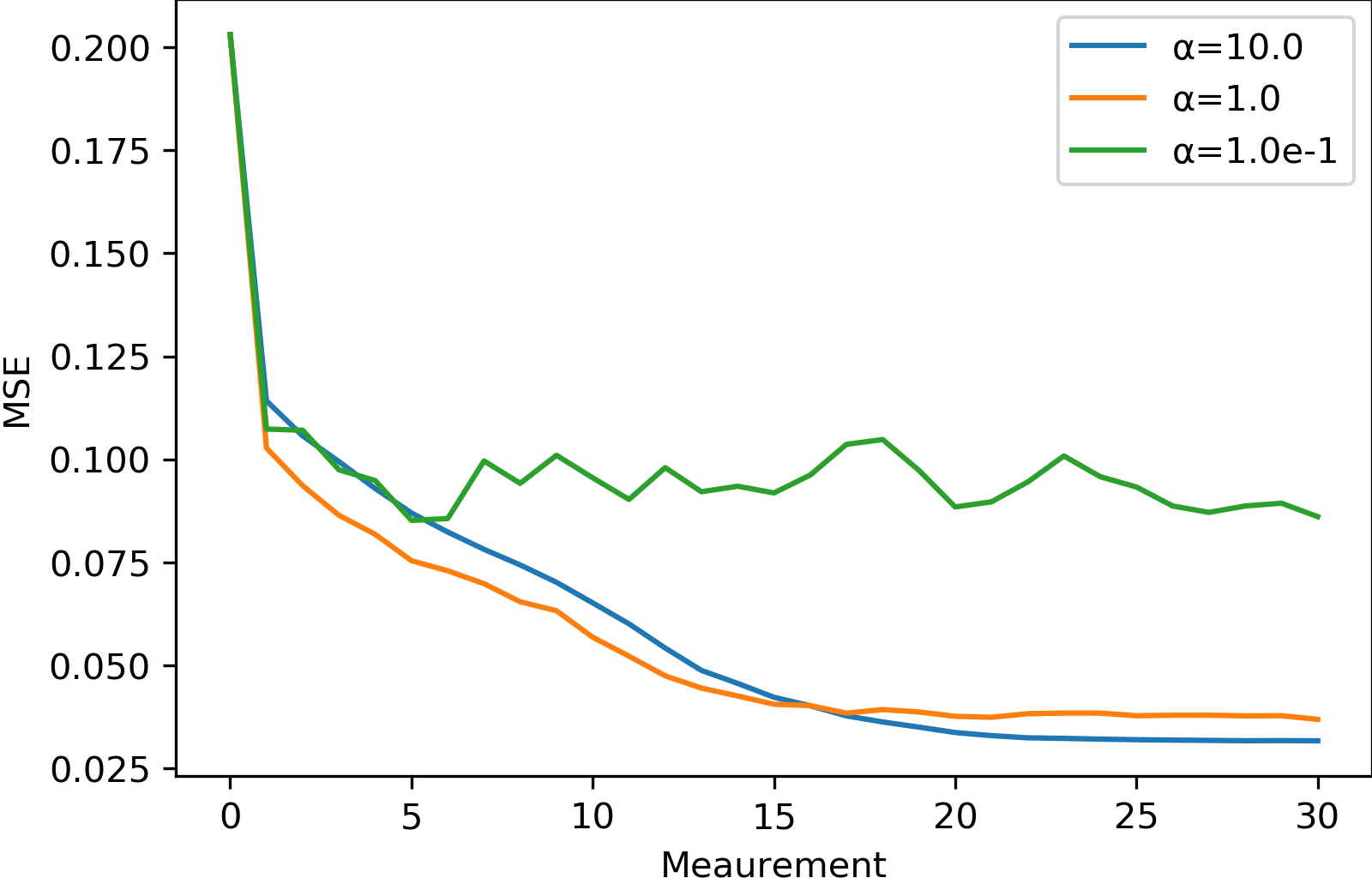}
  \subcaption{$B_1$, $k=1$, error graph}
 \end{minipage}
\end{tabular}
\begin{tabular}{c}
\hspace{-2.5cm}
 \begin{minipage}[b]{0.4\linewidth}
  \centering
  \includegraphics[keepaspectratio, scale=0.45]
  {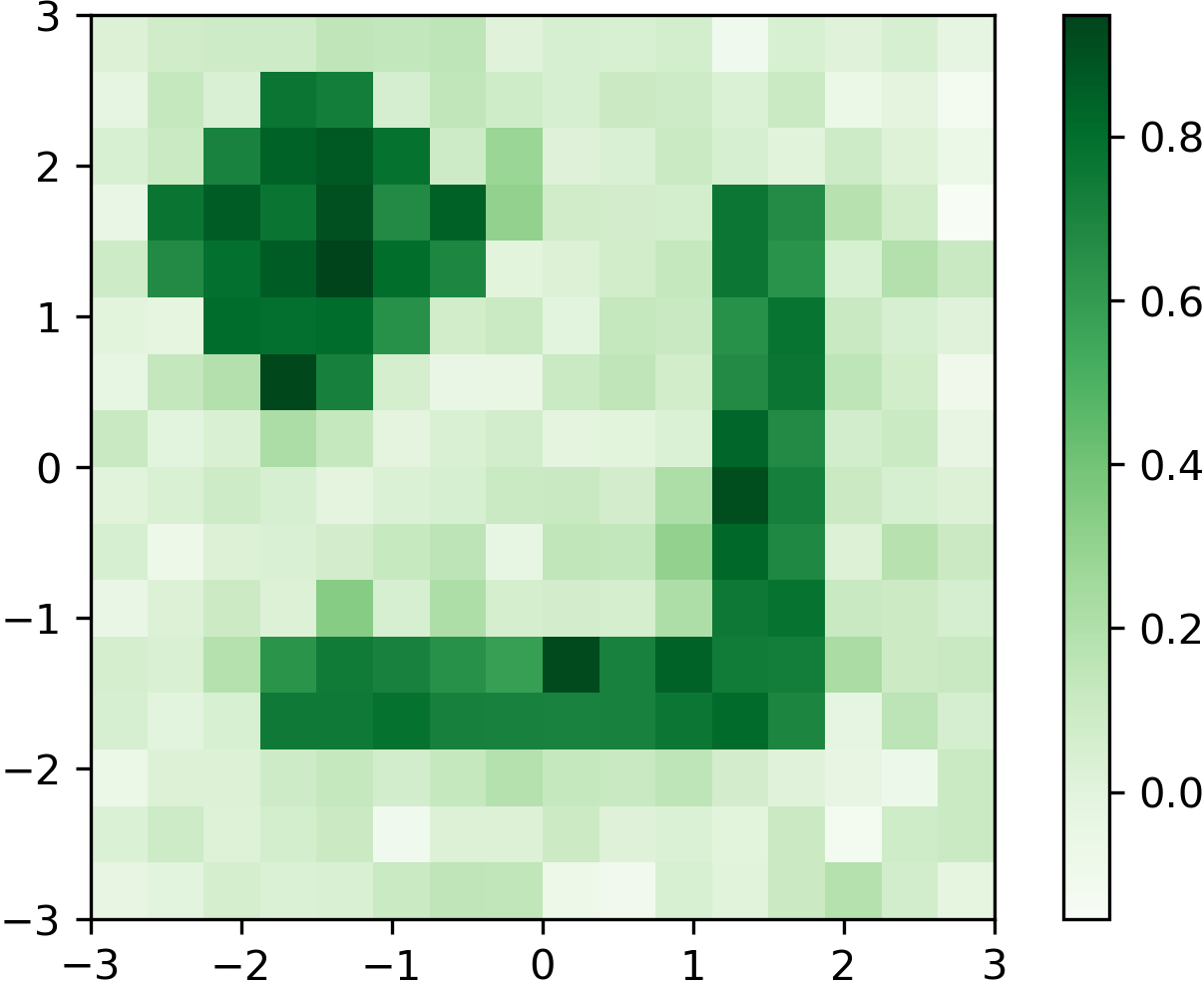}
  \subcaption{$B_2$, $k=5$, $\alpha=10$, $n=30$}
 \end{minipage}
 \begin{minipage}[b]{0.4\linewidth}
  \centering
  \includegraphics[keepaspectratio, scale=0.45]
  {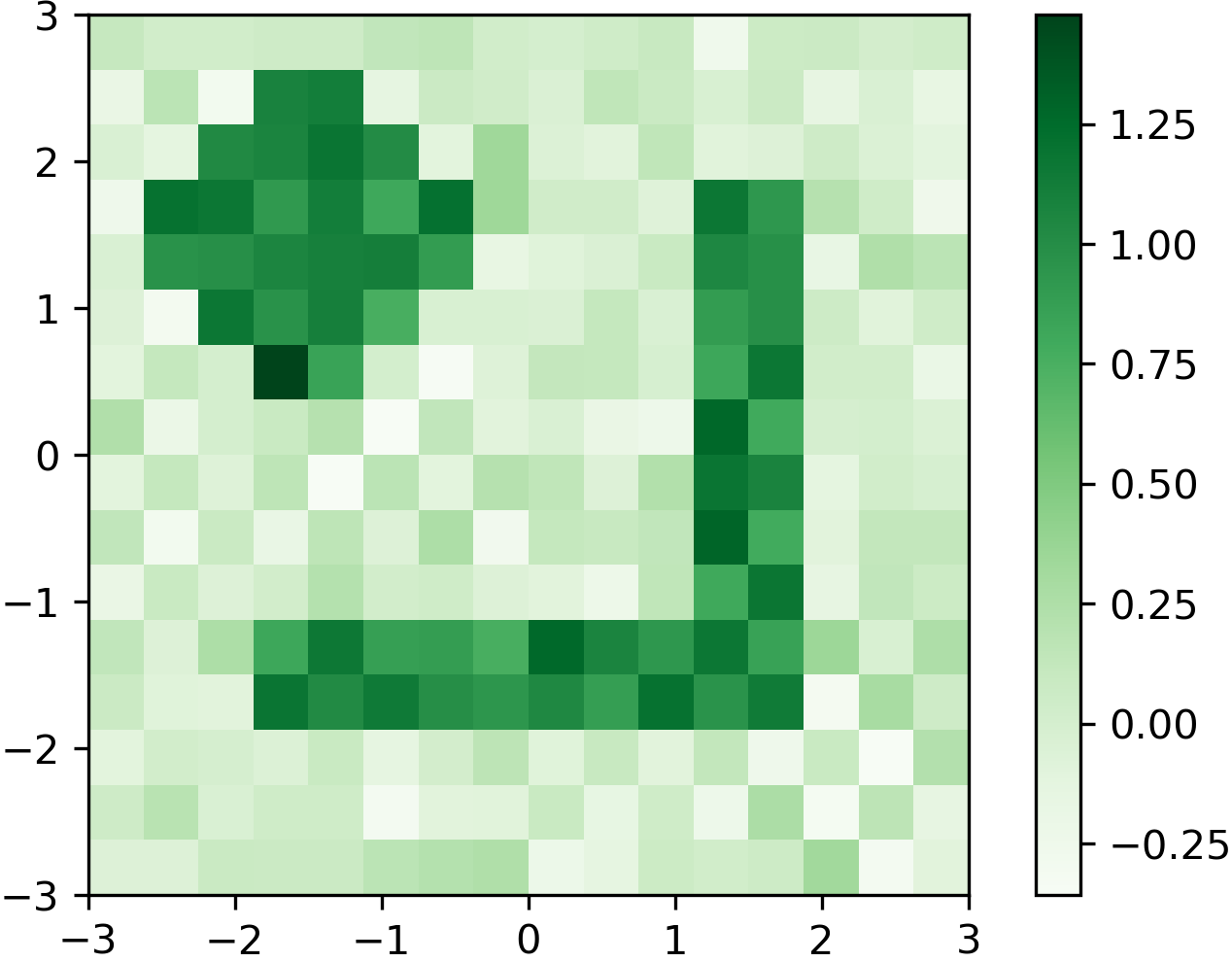}
  \subcaption{$B_2$, $k=5$, $\alpha=1\mathrm{e}-1$, $n=30$}
 \end{minipage}
 \begin{minipage}[b]{0.4\linewidth}
  \centering
  \includegraphics[keepaspectratio, scale=0.45]
  {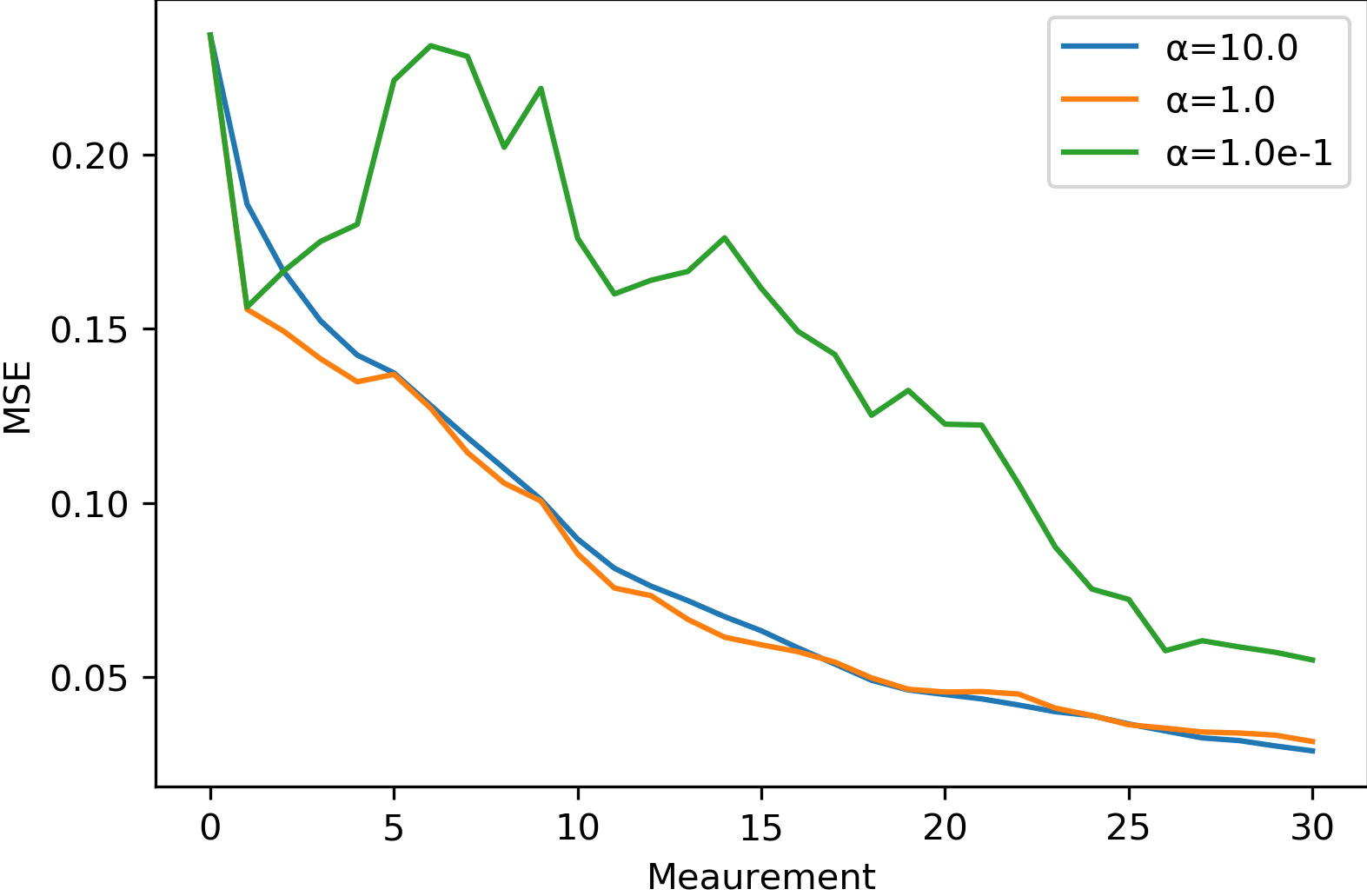}
  \subcaption{$B_2$, $k=5$, error graph}
 \end{minipage}
\end{tabular}
\begin{tabular}{c}
\hspace{-2.5cm}
 \begin{minipage}[b]{0.4\linewidth}
  \centering
  \includegraphics[keepaspectratio, scale=0.45]
  {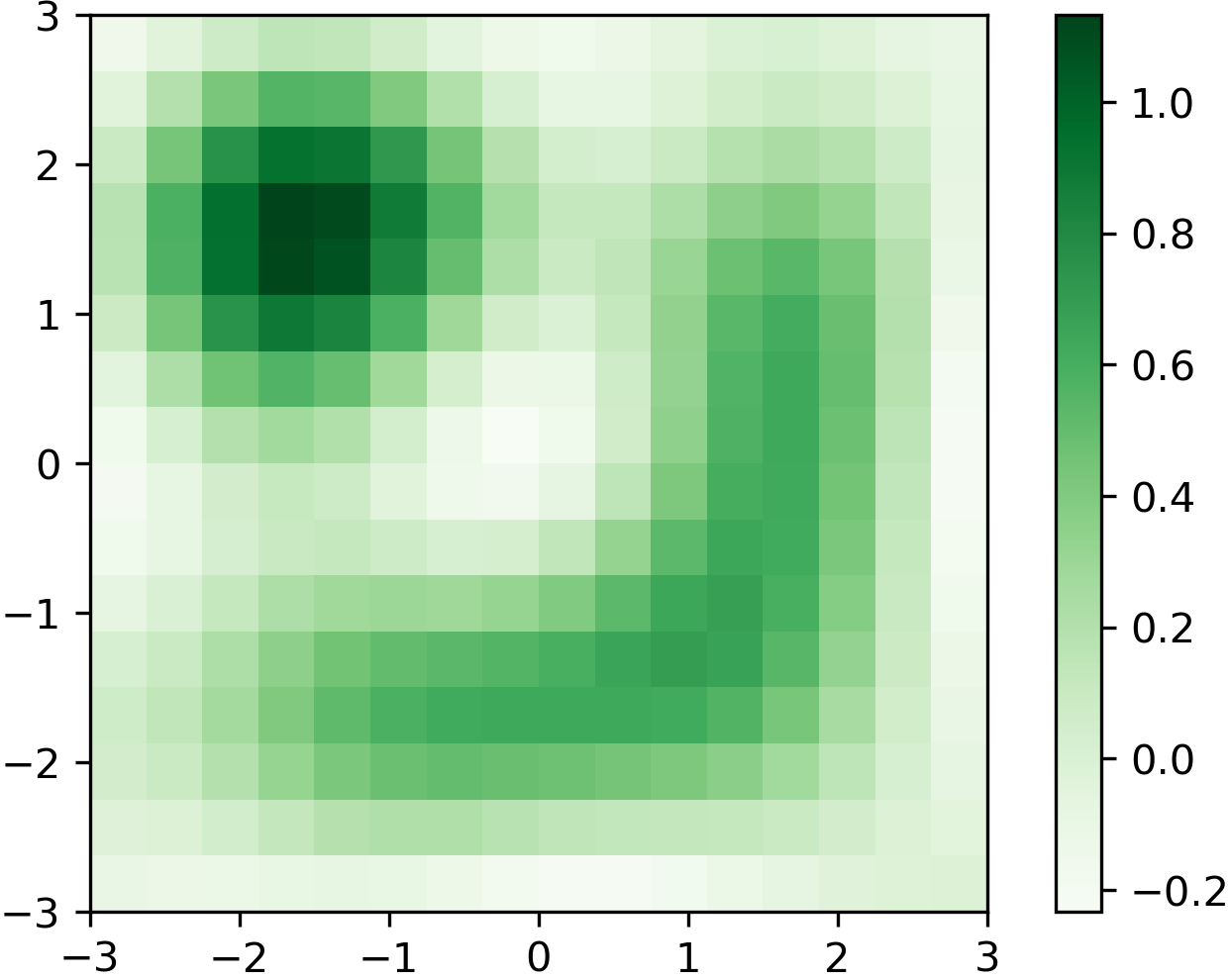}
  \subcaption{$B_2$, $k=1$, $\alpha=10$, $n=30$}
 \end{minipage}
 \begin{minipage}[b]{0.4\linewidth}
  \centering
  \includegraphics[keepaspectratio, scale=0.45]
  {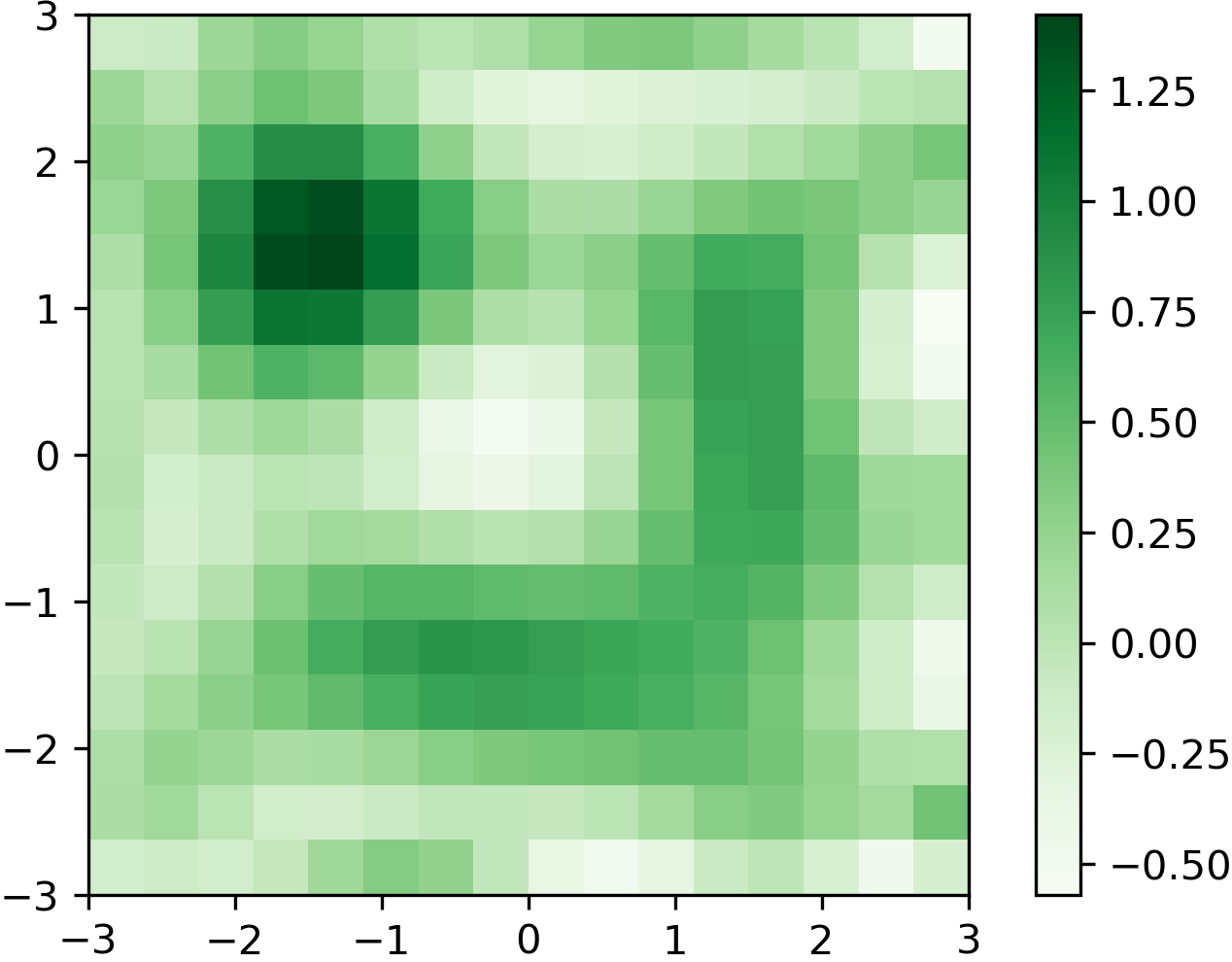}
  \subcaption{$B_2$, $k=1$, $\alpha=1\mathrm{e}-1$, $n=30$}
 \end{minipage}
 \begin{minipage}[b]{0.4\linewidth}
  \centering
  \includegraphics[keepaspectratio, scale=0.45]
  {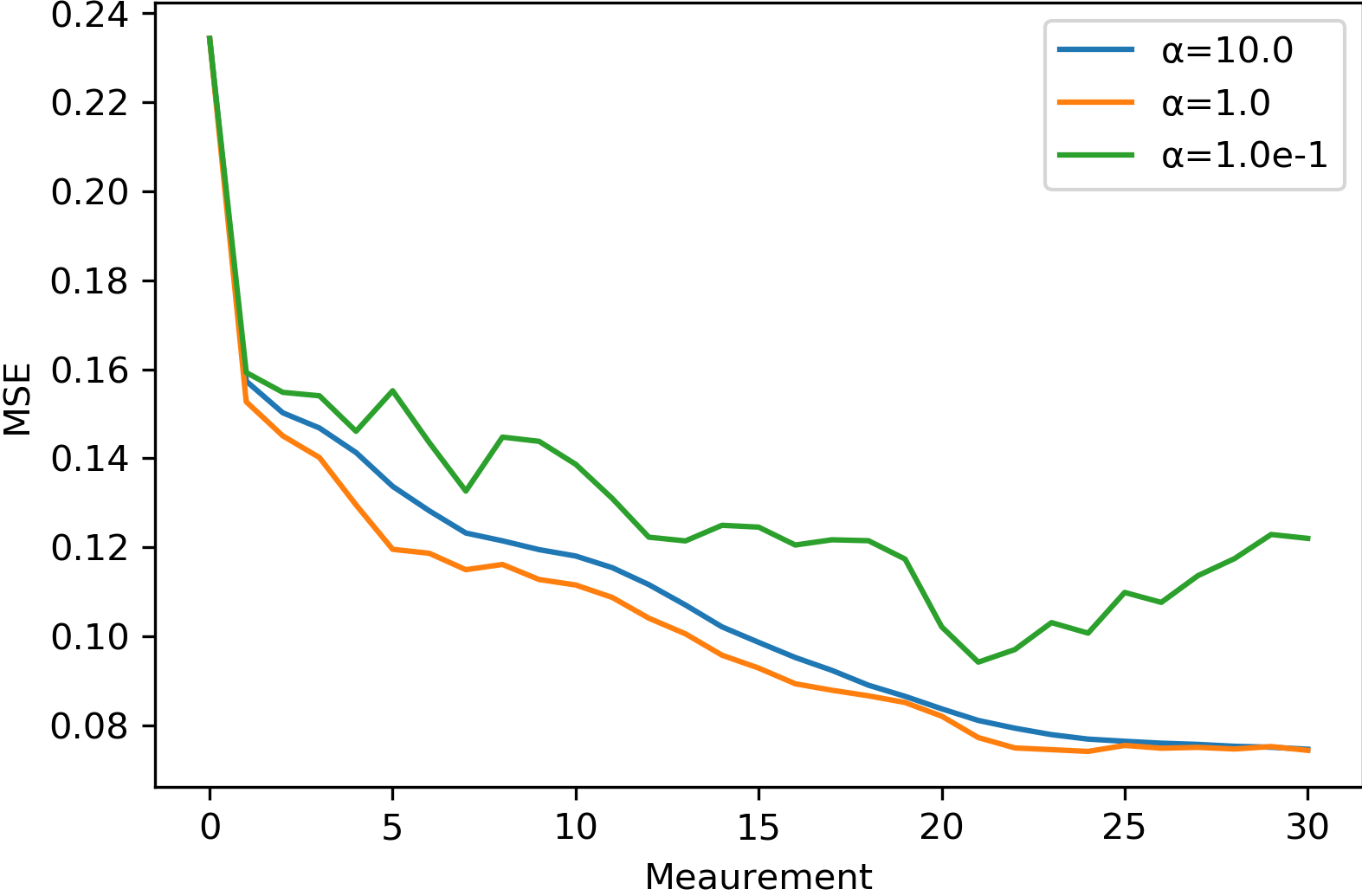}
  \subcaption{$B_2$, $k=1$, error graph}
 \end{minipage}
\end{tabular}
\caption{KF reconstruction for different $k$ and $\alpha$ (nosiy $\sigma= 0.5$)}
\label{KFreconstructionnosiy}
\end{figure}
\begin{figure}[h]
  \centering
  \includegraphics[keepaspectratio, scale=0.7]
  {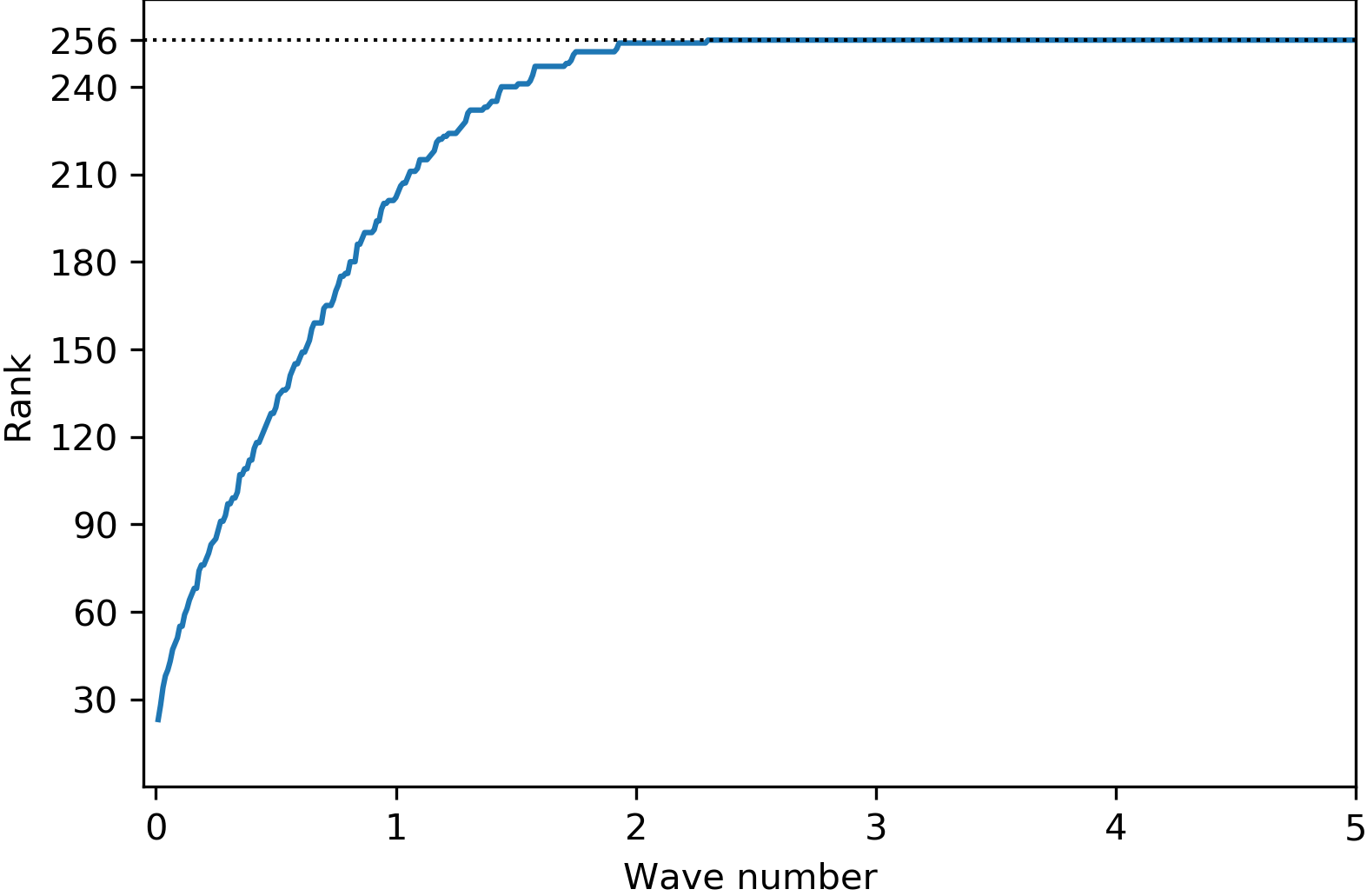}
 \caption{the graph of the rank of $\vec{\mathcal{F}}_{B} $}
\label{rank}
\end{figure}

\end{document}